\documentclass[a4paper,12pt]{article}
\usepackage{latexsym}
\usepackage{amsmath}
\usepackage{amssymb}
\usepackage{enumerate}
\usepackage{mathrsfs}

\usepackage{theorem}
\newtheorem{theorem}{Theorem}[section]
\newtheorem{proposition}[theorem]{Proposition}
\newtheorem{lemma}[theorem]{Lemma}
\newtheorem{corollary}[theorem]{Corollary}
\newtheorem{claim}[theorem]{Claim}

\theorembodyfont{\rmfamily}
\newtheorem{proof}{\textmd{\textit{Proof.}}}

\newtheorem{remark}[theorem]{Remark}
\newtheorem{example}[theorem]{Example}
\newtheorem{definition}[theorem]{Definition}

\makeatletter

\@addtoreset{equation}{section}
\makeatother

\newcommand{\qedd}{\hfill \Box}
\newcommand{\ve}{\varepsilon}
\newcommand{\del}{\partial}
\newcommand{\lra}{\longrightarrow}
\renewcommand{\det}{\ensuremath{\mathrm{det}}}
\newcommand{\grad}{\nabla\!\!_-}

\newcommand{\N}{\ensuremath{\mathbb{N}}}
\newcommand{\R}{\ensuremath{\mathbb{R}}}

\newcommand{\bv}{\ensuremath{\mathbf{v}}}
\newcommand{\cF}{\ensuremath{\mathcal{F}}}
\newcommand{\cL}{\ensuremath{\mathcal{L}}}
\newcommand{\cN}{\ensuremath{\mathcal{N}}}
\newcommand{\cP}{\ensuremath{\mathcal{P}}}
\newcommand{\cT}{\ensuremath{\mathcal{T}}}
\newcommand{\bJ}{\ensuremath{\mathbf{J}}}
\newcommand{\CD}{\ensuremath{\mathsf{CD}}}

\def\diam{\mathop{\mathrm{diam}}\nolimits}
\def\vol{\mathop{\mathrm{vol}}\nolimits}
\def\area{\mathop{\mathrm{area}}\nolimits}
\def\div{\mathop{\mathrm{div}}\nolimits}
\def\Hess{\mathop{\mathrm{Hess}}\nolimits}
\def\supp{\mathop{\mathrm{supp}}\nolimits}
\def\Ent{\mathop{\mathrm{Ent}}\nolimits}
\def\Ric{\mathop{\mathrm{Ric}}\nolimits}
\def\Sym{\mathop{\mathrm{Sym}}\nolimits}
\def\loc{\mathop{\mathrm{loc}}\nolimits}
\def\ac{\mathop{\mathrm{ac}}\nolimits}
\def\trace{\mathop{\mathrm{trace}}\nolimits}

\title{Displacement convexity\\ of generalized relative entropies}
\author{Shin-ichi Ohta\thanks{Department of Mathematics, Kyoto University,
Kyoto 606-8502, Japan ({\sf sohta@math.kyoto-u.ac.jp}) \&
Max-Planck-Intitut f\"ur Mathematik, Vivatsgasse 7, 53111 Bonn, Germany;
Supported in part by the Grant-in-Aid for Young Scientists (B) 20740036.}
and Asuka Takatsu\thanks{Graduate School of Mathematics, Nagoya University,
Nagoya 464-8602, Japan ({\sf takatsu@math.nagoya-u.ac.jp}) \&
Institut des Hautes \'Etudes Scientifiques, Bures-sur-Yvette 91440, France;
Supported in part by Research Fellowships of the JSPS for Young Scientists
\& by JSPS-IH\'ES (EPDI) Fellowship.}}
\date{\today}
\begin{document}

\maketitle

\begin{abstract}
We investigate the $m$-relative entropy, which stems from the Bregman divergence,
on weighted Riemannian and Finsler manifolds.
We prove that the displacement $K$-convexity of the $m$-relative entropy
is equivalent to the combination of the nonnegativity of the weighted Ricci curvature
and the $K$-convexity of the weight function.
We use this to show appropriate variants of the Talagrand, HWI and
the logarithmic Sobolev inequalities, as well as the concentration of measures.
We also prove that the gradient flow of the $m$-relative entropy produces
a solution to the porous medium equation or the fast diffusion equation.
\end{abstract}

\tableofcontents

\section{Introduction}

The {\it displacement convexity} of a functional on the space of probability measures
was introduced in McCann's influential paper~\cite{Mc1} as the covexity along geodesics
with respect to the $L^2$-Wasserstein distance.
Recent astonishing development of optimal transport theory reveals that
the displacement convexity of entropy-type functionals plays important roles
in the theory of partial differential equations, probability theory
and differential geometry (see \cite{AGS}, \cite{Vi1}, \cite{Vi2} and the references therein).
For instance, on a compact Riemannian manifold $(M,g)$ equipped with
the Riemannian volume measure $\vol_g$, the gradient flow of the relative entropy
$\Ent_{\vol_g}$ (see $(\ref{eq:rent})$ for definition) in the $L^2$-Wasserstein space
$(\cP(M),W_2)$ produces a weak solution to the heat equation
(\cite{Ogra}, \cite{GO}, \cite[Chapter~23]{Vi2}).
Then the displacement $K$-convexity of $\Ent_{\vol_g}$ for some $K \in \R$
(denoted by $\Hess\Ent_{\vol_g} \ge K$ for short) implies the $K$-contraction property
\[ W_2\big( p(t,x,\cdot)\vol_g,p(t,y,\cdot)\vol_g \big) \le e^{-Kt}d(x,y),
 \qquad x,y \in M, \]
of the heat kernel $p:(0,\infty) \times M \times M \lra (0,\infty)$
(and vice versa, \cite{vRS}), where $d$ is the Riemannian distance.
The condition $\Hess\Ent_{\vol_g} \ge K$ is called the
{\it curvature-dimension condition} $\CD(K,\infty)$ and known to be equivalent
to the lower Ricci curvature bound $\Ric \ge K$ (\cite{vRS}).
There is a rich theory on general metric measure spaces satisfying $\CD(K,\infty)$
(\cite{StI}, \cite{LV2}, \cite[Part~III]{Vi2}).
Especially, $\CD(K,\infty)$ with $K>0$ is an important condition which yields,
among others, the logarithmic Sobolev inequality and the normal concentration
of measures (a kind of large deviation principle).

The curvature-dimension condition is generalized to $\CD(K,N)$
for each $K \in \R$ and $N \in (1,\infty]$, and then $\CD(K,N)$
is equivalent to the lower bound of the weighted Ricci curvature
$\Ric_N \ge K$ of a weighted Riemannian manifold $(M,\omega)$,
where $\omega$ is a conformal deformation of $\vol_g$ (\cite{StII}, \cite{LV1},
see \eqref{eq:RicN} for the definition of $\Ric_N$).
However, $\CD(K,N)$ with $N<\infty$ is written as a simple convexity condition
only when $K=0$ (and it causes some difficulties when $K \neq 0$, see \cite{BS}).
Precisely, $\CD(0,N)$ is defined as the convexity of the R\'enyi entropy
$S_N$ (see $(\ref{eq:Ren})$ for definition),
while $\CD(K,N)$ with $K \neq 0$ is a more subtle inequality
involving the integrand of $S_N$.
Sturm has shown in \cite[Theorem~1.7]{Stcon} that there are (by no means unique)
functionals whose displacement $K$-convexity is equivalent to the combination of
$\Ric \ge K$ and $\dim \le N$ for unweighted Riemannian manifolds,
but it is unclear how this observation relates to $\CD(K,N)$.

In this article, we introduce and consider a different kind of relative entropy
$H_m(\cdot|\nu)$ for $m \in [(n-1)/n,1) \cup (1,\infty)$
---we call this the {\it $m$-relative entropy}---
which is related to, but different from $S_N$.
Here $\nu=\exp_m(-\Psi) \omega$ is a fixed conformal deformation
of $\omega$, and $\exp_m$ is the $m$-exponential function
(see Subsection~\ref{ssc:Nm}).
Our definition of $H_m(\cdot|\nu)$ stems from the Bregman divergence
in information theory/geometry which is closely related to
the Tsallis and R\'enyi entropies (see Subsection~\ref{ssc:ent}).
Roughly speaking, $H_m(\mu|\nu)$ is defined as
\[ H_m(\mu|\nu)=\frac{1}{m(m-1)} \int_M
 \{ \rho^m-m\rho\sigma^{m-1}+(m-1)\sigma^m \} \,d\omega, \]
for $\mu=\rho\omega$ and $\nu=\sigma\omega$
(see Definition~\ref{df:Hm} for the precise definition).
We can regard $H_m(\mu|\nu)$ as representing the difference
between $\mu$ and $\nu$.
Taking the limit as $m$ tends to $1$ recovers the usual relative entropy $\Ent_{\nu}$
(or the Kullback-Leibler divergence $H(\cdot|\nu)$).
Our results will guarantee that $H_m(\cdot|\nu)$ is a natural and important object.

Our first main theorem asserts that $\Hess H_m(\cdot|\nu) \ge K$
in $(\cP^2(M),W_2)$ is equivalent to the combination of $\Ric_N \ge 0$ with
$N=1/(1-m)$ and $\Hess\Psi \ge K$, where $\Ric_N$ is of $(M,\omega)$
(Theorem~\ref{th:mCD}).
We remark that $N$ can be negative, such $\Ric_N$ is not previously studied
and would be of independent interest.
It is also interesting to obtain split curvature bound/convexity conditions
from a single convexity condition of the entropy.
Then, according to the technique similar to the curvature-dimension condition,
we show that $\Ric_N \ge 0$ and $\Hess\Psi \ge K$ imply appropriate variants of
the Talagrand, HWI, logarithmic Sobolev and the global Poincar\'e inequalities
(Propositions~\ref{pr:Ta}, \ref{pr:Poin}, Theorem~\ref{th:LS}),
and also the concentration of measures (Theorem~\ref{th:conc}, Proposition~\ref{pr:conc}).
Furthermore, the gradient flow of $H_m(\cdot|\nu)$ produces a weak solution
to the porous medium equation (for $m>1$)
or the fast diffusion equation (for $m<1$) of the form
\[ \frac{\del \rho}{\del t}
 =\frac{1}{m}\Delta^{\omega}(\rho^m)+\div_{\omega}(\rho\nabla\Psi), \]
where $\Delta^{\omega}$ and $\div_{\omega}$ are the Laplacian and
the divergence associated with the measure $\omega$ (Theorem~\ref{th:gf}).
Among others, we shall follow the metric geometric way of interpreting
this coincidence as in \cite{Ogra}, \cite{GO}.
Most results hold true also for Finsler manifolds thanks to the theory
developed in \cite{Oint} and \cite{OS} (see Section~\ref{sc:Fins}).

We comment on former related work on this kind of entropy.
On unweighted Riemannian manifolds, Sturm showed a similar
characterization of the displacement $K$-convexity of a class of entropies
(or free energies) including $H_m$ (\cite[Theorem~1.3]{Stcon}).
We generalize this to weighted Riemannian (and even Finsler) manifolds,
and then $\Ric$ is replaced with $\Ric_N$ (this is natural but nonobvious).
Also our treatment of singular measures is more precise than \cite{Stcon}.
Gradient flow from the view of Wasserstein geometry has been investigated
by Otto~\cite{Ot} in the Euclidean case, and by Villani~\cite[Chapters~23, 24]{Vi2}
in the weighted Riemannian case in a different manner from ours.
Functional inequalities related to the convexity of the weight $\Psi$
were studied in \cite{AGK}, \cite{CGH} and \cite{Ta2} in Euclidean spaces
(see also \cite[Remark~1.1]{Stcon} and \cite[Chapters~24, 25]{Vi2}).
The concentration of measures seems new even in the Euclidean setting.

The organization of the article is as follows.
After preliminaries, we introduce the $m$-relative entropy $H_m(\cdot|\nu)$
in Section~\ref{sc:Hm}, and show that $\Hess H_m(\cdot|\nu) \ge K$
is equivalent to $\Hess\Psi \ge K$ with $\Ric_N \ge 0$ in Section~\ref{sc:dc}.
Using this equivalence, we obtain several functional inequalities in
Section~\ref{sc:func}, and the concentration of measures in Section~\ref{sc:conc}.
Section~\ref{sc:gf} is devoted to the study of the gradient flow of $H_m(\cdot|\nu)$.
Finally, we treat the Finsler case in Section~\ref{sc:Fins}.

\section{Preliminaries}\label{sc:pre}

Throughout the article except the last section,
$(M,g)$ will be a complete, connected $n$-dimensional $C^{\infty}$-Riemannian manifold
and $d$ stands for the Riemannian distance of $M$.
For simplicity and since we are interested in the role of curvature bounds,
we will always assume $n \ge 2$.
Denote by $B(x,r)$ the open ball of center $x \in M$ and radius $r>0$,
i.e., $B(x,r)=\{ y \in M \,|\, d(x,y)<r \}$.
See, e.g., \cite{Ch} for the basics of Riemannian geometry.

\subsection{Weighted Ricci curvature}

We fix a conformal change $\omega=e^{-\psi}\vol_g$, with $\psi \in C^{\infty}(M)$,
of the Riemannian volume measure $\vol_g$ as our base measure.
Given a unit vector $v \in T_xM$ and $N \in (-\infty,0) \cup (n,\infty)$,
we define the {\it weighted Ricci curvature} by
\begin{equation}\label{eq:RicN}
\Ric_N(v) :=\Ric(v)+\Hess\psi(v,v) -\frac{\langle \nabla\psi,v\rangle^2}{N-n}.
\end{equation}
We also set
\[ \Ric_n(v):= \left\{ \begin{array}{cl}
 \Ric(v)+\Hess\psi(v,v) & {\rm if}\ \langle \nabla\psi,v\rangle=0, \\
 -\infty & {\rm otherwise}.
\end{array} \right. \]
Observe that, if $\psi$ is constant, then $\Ric_N(v)$ coincides with $\Ric(v)$ for all $N$.

\begin{remark}\label{rm:N<0}
We usually consider $\Ric_N$ only for $N \in [n,\infty]$
(where $\Ric_{\infty}(v)=\Ric(v)+\Hess\psi(v,v)$ is the {\it Bakry-\'Emery tensor},
see \cite{BE}, \cite{Qi}, \cite{Lo}),
and then it enjoys the monotonicity: $\Ric_N(v) \le \Ric_{N'}(v)$ for $N<N'$.
Admitting $N<0$ violates this monotonicity,
but we abuse this notation for brevity.
The reason why we consider this range of $N$ will be seen in $(\ref{eq:N})$.
\end{remark}

As we mentioned in the introduction, $\Ric_N \ge K$ for $K \in \R$ and
$N \ge n$ is equivalent to Sturm's {\it curvature-dimension condition} $\CD(K,N)$.
Spaces satisfying $\CD(K,N)$ behave like a space with ``dimension $\le N$
as well as Ricci curvature $\ge K$'' (see \cite{StII}, \cite{LV1}, \cite[Part~III]{Vi2}).

\subsection{Generalized exponential functions and Gaussian measures}\label{ssc:Nm}

We briefly recall the {\it $m$-calculus}, see \cite{Ts2} for further discussion.
We introduce a parameter $m$ such that
\[ m \in [(n-1)/n,1) \cup (1,\infty). \]
We sometimes eliminate the special case $m=1/2$ with $n=2$ (Section~\ref{sc:func})
or restrict ourselves to $m \le 2$ (Sections~\ref{sc:conc}, \ref{sc:gf}).
We set
\begin{equation}\label{eq:N}
N=N(m):=1/(1-m) \in (-\infty,0) \cup [n,\infty).
\end{equation}
Define the {\it $m$-logarithmic function} by
\[ \ln_m(t):=\frac{t^{m-1}-1}{m-1}\quad {\rm for}\
 \left\{ \begin{array}{cl}
 t>0\ & {\rm if}\ m<1, \\ t \ge 0\ & {\rm if}\ m>1.
 \end{array} \right. \]
Note that $\ln_m$ is monotone increasing and that the image of
$\ln_m$ is $(-\infty,1/(1-m))$ if $m<1$; $[-1/(m-1),\infty)$ if $m>1$.
We define the {\it $m$-exponential function} $\exp_m$ as the inverse of $\ln_m$,
namely
\[ \exp_m(t):=\{ 1+(m-1)t \}^{1/(m-1)}\quad {\rm for}\
 \left\{ \begin{array}{cl}
 t \in (-\infty,1/(1-m)) & {\rm if}\ m<1, \\
 t \in [-1/(m-1),\infty) & {\rm if}\ m>1.
\end{array} \right. \]
For the sake of simplicity, we set $\exp_m(t):=0$ for $m>1$ and $t<-1/(m-1)$.
We also define
\[ e_m(t):=t\ln_m(t)=\frac{t^m-t}{m-1}\quad {\rm for}\ t>0, \qquad e_m(0):=0. \]
Observe that
\[ \lim_{m \to 1} \ln_m(t)=\ln(t), \qquad \lim_{m \to 1} \exp_m(t)=e^t,
 \qquad \lim_{m \to 1} e_m(t)=t\ln(t). \]

\begin{remark}\label{rm:q}
(1) Taking $m<1$ and $m>1$ gives rise to qualitatively different phenomena
(see Lemma~\ref{lm:K>0}, Example~\ref{ex:Nm} for instances).
Nonetheless, most of our results will cover both cases.

(2) In some notations, it is common to use the parameter $q=2-m$
instead of $m$ (e.g., $\exp_q$ and $q$-Gaussian measures).
In the present paper, however, we shall use only $m$ for brevity.
\end{remark}

Using $\exp_m$ and the base measure $\omega=e^{-\psi}\vol_g$,
we will fix another measure
\[ \nu=\sigma\omega:= \exp_m(-\Psi)\omega \]
as our reference measure, where $\Psi \in C(M)$ such that
$\Psi>-1/(1-m)$ if $m<1$.
Note that the two weights $e^{-\psi}$ and $\exp_m(-\Psi)$
involve different kinds of exponential function, so that they can not be combined.
For later convenience, we set
\begin{equation}\label{eq:M0}
M_0:= \left\{ \begin{array}{cl}
 M & {\rm for}\ m<1, \\
 \Psi^{-1}\big( (-\infty,1/(m-1)) \big) & {\rm for}\ m>1,
\end{array} \right.
\end{equation}
and assume that $M_0$ is nonempty.
Note that $\supp\nu=\overline{M_0}$ holds in both cases.
We shall study how the convexity of $\Psi$ has an effect on
the geometric and analytic structures of $(M,\nu)$.

\begin{definition}[$K$-convexity]\label{df:Psi}
Given $K \in \R$, we say that $\Psi$ is {\it $K$-convex in the weak sense},
denoted by $\Hess\Psi \ge K$ for short, if any two points $x,y \in M$ admit
a minimal geodesic $\gamma:[0,1] \lra M$ from $x$ to $y$ along which
\begin{equation}\label{eq:Psi}
\Psi\big( \gamma(t) \big) \le (1-t)\Psi(x) +t\Psi(y) -\frac{K}{2}(1-t)td(x,y)^2
\end{equation}
holds for all $t \in [0,1]$.
\end{definition}

Note that this is equivalent to saying that $(\ref{eq:Psi})$ holds
along any minimal geodesic $\gamma$ between $x$ and $y$,
for $\gamma|_{[\ve,1-\ve]}$ is a unique minimal geodesic for all $\ve>0$
and $\Psi$ is continuous.

\begin{remark}\label{rm:multi}
Consider a different presentation
$\nu=(c\sigma)(c^{-1}\omega)=:\tilde{\sigma}\tilde{\omega}$ of $\nu$
for some constant $c>0$.
Then the weighted Ricci curvature $\Ric_N$ is unchanged, while
\begin{align*}
\tilde{\sigma}
&=c\exp_m(-\Psi) =\{ c^{m-1}-(m-1)c^{m-1}\Psi \}^{1/(m-1)} \\
&= \bigg\{ 1-(m-1)\bigg( c^{m-1}\Psi-\frac{c^{m-1}-1}{m-1} \bigg) \bigg\}^{1/(m-1)}
 =: \exp_m(-\widetilde{\Psi})
\end{align*}
and hence $\Hess\widetilde{\Psi}=c^{m-1}\Hess\Psi$.
\end{remark}

Sections~\ref{sc:func}, \ref{sc:conc} will be concerned with the case where
$\Hess\Psi \ge K>0$ as well as $\Ric_N \ge 0$.
In such a situation, it turns out that $\nu$ has finite total mass.
Here we give explicit estimates for later use (in Section~\ref{sc:conc}).

\begin{lemma}\label{lm:K>0}
Assume that $\Hess\Psi \ge K$ holds for some $K>0$,
and take a unique minimizer $x_0 \in M$ of $\Psi$.
\begin{enumerate}[{\rm (i)}]
\item
If $m<1$ and $\Ric_N \ge 0$, then $\sigma \in L^c(M,\omega)$
for all $c \in (1/2,1]$, in particular, $\nu(M)<\infty$.
Moreover, we have
\[ \int_M \sigma^c \,d\omega \le C_1^{1-c}\nu(M)^c
 +C_2K^{c/(m-1)} \]
for some $C_1=C_1(\omega)>0$ and $C_2=C_2(m,c,\omega)>0$.

\item
If $m<1$ and $\Ric_N \ge 0$, then $\int_M d(x_0,x)^p \,d\nu<\infty$
for all $p \in [1,1/(1-m))$.

\item
If $m>1$, then $M_0$ and $\supp\nu$ are convex in the sense
that any pair of points in $M_0$ or $\supp\nu$ is connected by a minimal geodesic
contained in $M_0$ or $\supp\nu$, respectively.
In addition, we have
\[ \supp\nu \subset \overline{B\bigg( x_0,\bigg\{ \frac{2}{K}
 \bigg( \frac{1}{m-1}-\Psi(x_0) \bigg) \bigg\}^{1/2} \bigg)}. \]
\end{enumerate}
\end{lemma}

\begin{proof}
By our assumption $\Hess\Psi \ge K>0$, we find a unique point $x_0 \in M_0$
such that $\Psi(x_0)=\inf_M \Psi$.
Then we deduce from $(\ref{eq:Psi})$ that
\[ \Psi\big( \gamma(1) \big)
 \ge \Psi(x_0)+\frac{K}{2}d\big( x_0,\gamma(1) \big)^2 \]
holds for all minimal geodesics $\gamma$ with $\gamma(0)=x_0$.
Thus we have
\begin{equation}\label{eq:K>0}
\sigma(x)=\exp_m\big( -\Psi(x) \big)
 \le \exp_m\bigg( -\Psi(x_0)-\frac{K}{2}d(x_0,x)^2 \bigg)
\end{equation}
for all $x \in M_0$.

(i) Denote by $\area_{\omega}(S(x_0,r))$ the area of the sphere
$S(x_0,r):=\{ x \in M \,|\, d(x_0,x)=r \}$ with respect to $\omega$.
Then $(\ref{eq:K>0})$ implies
\[ \int_M \sigma^c \,d\omega \le \int_{B(x_0,1)} \sigma^c \,d\omega
 +\int_1^{\infty} \exp_m\bigg( -\Psi(x_0)-\frac{K}{2}r^2 \bigg)^c
 \area_{\omega}\big( S(x_0,r) \big) \,dr. \]
On the one hand, it follows from $\Ric_N \ge 0$ that, for $r \ge 1$,
\[ \area_{\omega}\big( S(x_0,r) \big)
 \le r^{N-1} \area_{\omega}\big( S(x_0,1) \big)
 =r^{m/(1-m)} \area_{\omega}\big( S(x_0,1) \big) \]
(cf.\ \cite[Theorem~2.3]{StII}).
Therefore we obtain, putting $a:=\exp_m(-\Psi(x_0))^{m-1}>0$,
\begin{align*}
&\int_1^{\infty} \exp_m\bigg( -\Psi(x_0)-\frac{K}{2}r^2 \bigg)^c
 \area_{\omega}\big( S(x_0,r) \big) \,dr \\
&\le \area_{\omega}\big( S(x_0,1) \big) \int_1^{\infty}
 \bigg\{ a+(1-m)\frac{K}{2}r^2 \bigg\}^{c/(m-1)} r^{m/(1-m)} \,dr \\
&= \area_{\omega}\big( S(x_0,1) \big) \int_1^{\infty}
 \bigg\{ ar^{-2}+(1-m)\frac{K}{2} \bigg\}^{c/(m-1)} r^{(m-2c)/(1-m)} \,dr \\
&\le \area_{\omega}\big( S(x_0,1) \big) \bigg\{ (1-m)\frac{K}{2} \bigg\}^{c/(m-1)}
 \int_1^{\infty} r^{(m-2c)/(1-m)} \,dr.
\end{align*}
As $c>1/2$, the most right-hand side coincides with
\[ \area_{\omega}\big( S(x_0,1) \big) \frac{(1-m)^{c/(m-1)+1}}{2c-1}
 \bigg( \frac{K}{2} \bigg)^{c/(m-1)}
 =:C_2(m,c,\omega) K^{c/(m-1)} <\infty. \]
On the other hand, as $\nu(M)<\infty$ is already observed,
the H\"older inequality and $c \le 1$ yield
\[ \int_{B(x_0,1)} \sigma^c \,d\omega
 \le \bigg( \int_{B(x_0,1)} \sigma \,d\omega \bigg)^c \omega\big( B(x_0,1) \big)^{1-c}
 \le \nu(M)^c \omega\big( B(x_0,1) \big)^{1-c}. \]
We set $C_1(\omega)=\omega(B(x_0,1))$ and complete the proof.

(ii) We similarly deduce from $\eqref{eq:K>0}$ and $\Ric_N \ge 0$ that
\begin{align*}
&\int_{M \setminus B(x_0,1)} d(x_0,x)^p \,d\nu(x) \\
&\le \int_1^{\infty} r^p \exp_m\bigg( -\Psi(x_0)-\frac{K}{2}r^2 \bigg)
 \area_{\omega}\big( S(x_0,r) \big) \,dr \\
&\le \area_{\omega}\big( S(x_0,1) \big) \bigg\{ (1-m)\frac{K}{2} \bigg\}^{1/(m-1)}
 \int_1^{\infty} r^{p+(m-2)/(1-m)} \,dr \\
&= \area_{\omega}\big( S(x_0,1) \big)
 \frac{(1-m)^{m/(m-1)}}{1-(1-m)p} \bigg( \frac{K}{2} \bigg)^{1/(m-1)}
 <\infty.
\end{align*}
We used $p<1/(1-m)$ to see $p+(m-2)/(1-m)<-1$.

(iii) Recall that $M_0=\Psi^{-1}((-\infty,1/(m-1)))$ and $\supp\nu=\overline{M_0}$.
Therefore $M_0$ and $\supp\nu$ are convex and
$(\ref{eq:K>0})$ shows the desired estimate.
$\qedd$
\end{proof}

Observe that the convexity of $M_0$ and $\supp\nu$ in Lemma~\ref{lm:K>0}(iii)
holds true also for $K=0$.

\begin{example}[$m$-Gaussian measures]\label{ex:Nm}
One fundamental and important example to which Lemma~\ref{lm:K>0} applies
is the {\it $m$-Gaussian measure} on $\R^n$ defined by
\begin{equation}\label{eq:Nm}
N_m(v,V)=\sigma dx := \frac{C_0}{(\det V)^{1/2}}
 \exp_m \bigg[ -\frac{C_1}{2}\langle x-v,V^{-1}(x-v) \rangle \bigg] dx,
\end{equation}
where $dx$ is the Lebesgue measure, a vector $v \in \R^n$ is the mean,
a positive-definite symmetric matrix $V \in \Sym^+(n,\R)$ is the covariance matrix,
and $C_0,C_1$ are positive constants depending only on $n$ and $m$ (see \cite{Ta2}).
Then clearly $\Hess\Psi =C_0^{m-1}(\det V)^{(1-m)/2} \cdot C_1V^{-1}$
(by taking Remark~\ref{rm:multi} into account) and hence
\[ \Hess\Psi \ge C_0^{m-1}C_1(\det V)^{(1-m)/2}\Lambda^{-1}>0, \]
where $\Lambda$ denotes the largest eigenvalue of $V$.
Note that $N_m(v,V)$ has unbounded and bounded support for
$m<1$ and $m>1$, respectively.
The family of $m$-Gaussian measures will play interesting roles in
Sections~\ref{sc:Hm}, \ref{sc:func}, \ref{sc:gf}.
\end{example}

\subsection{Wasserstein geometry}

We very briefly recall some fundamental facts in optimal transport theory and Wasserstein geometry.
We refer to \cite{Vi1}, \cite{Vi2} for basics as well as recent diverse development
of them.

Let $(X,d)$ be a complete, separable metric space.
A rectifiable curve $\gamma:[0,1] \lra X$ is called a {\it geodesic}
if it is locally minimizing and has a constant speed,
we say that $\gamma$ is {\it minimal} if it is globally minimizing
(i.e., $d(\gamma(s),\gamma(t))=|s-t|d(\gamma(0),\gamma(1))$
for all $s,t \in [0,1]$).
If any two points in $X$ is connected by a minimal geodesic,
then $(X,d)$ is called a {\it geodesic space}.

We denote by $\cP(X)$ the set of all Borel probability measures on $X$,
and by $\cP^p(X) \subset \cP(X)$ with $p \ge 1$ the subset consisting of
measures $\mu$ of finite $p$-th moment, that is,
$\int_X d(x,y)^p \,d\mu(y) <\infty$ for some (and hence all) $x \in X$.
Clearly $\cP^p(X)=\cP(X)$ if $X$ is bounded.
Given $\mu,\nu \in \cP(X)$, a probability measure $\pi \in \cP(X \times X)$
is called a {\it coupling} of $\mu$ and $\nu$ if its projections coincides with $\mu$ and $\nu$,
namely $\pi(A \times X)=\mu(A)$ and $\pi(X \times A)=\nu(A)$
hold for any Borel set $A \subset X$.
We define the {\it $L^p$-Wasserstein distance} between $\mu,\nu \in \cP^p(X)$ by
\[ W_p(\mu,\nu):=\inf_{\pi}
 \bigg( \int_{X \times X} d(x,y)^p \,d\pi(x,y) \bigg)^{1/p}, \]
where $\pi$ runs over all couplings of $\mu$ and $\nu$.
We call $\pi$ an {\it optimal coupling} if it attains the infimum above.
We remark that $W_p(\mu,\nu)$ is finite since $\mu,\nu \in \cP^p(X)$,
and it is indeed a distance of $\cP^p(X)$.
The metric space $(\cP^p(X),W_p)$ is called the
{\it $L^p$-Wasserstein space} over $X$.
If $X$ is compact, then $(\cP(X),W_p)$ is also compact and the topology
induced from $W_p$ coincides with the weak topology.

We will consider only the case of $p=2$ that is suitable and important
for applications in Riemannian geometry.
A minimal geodesic between $\mu,\nu \in \cP^2(X)$ amounts to
an optimal way of transporting $\mu$ to $\nu$ with respect to the quadratic cost $d(x,y)^2$.
Then it is natural to expect that such an optimal transport is performed along
minimal geodesics in $X$, that is indeed the case as seen in the following proposition.
We denote by $\Gamma(X)$ the set of all minimal geodesics
$\gamma:[0,1] \lra X$ endowed with the topology induced from the distance
$d_{\Gamma(X)}(\gamma,\eta):=\sup_{t \in [0,1]}d(\gamma(t),\eta(t))$.
For $t \in [0,1]$, define the {\it evaluation map} $e_t:\Gamma(X) \lra X$
as $e_t(\gamma):=\gamma(t)$, and observe that each $e_t$ is $1$-Lipschitz.

\begin{proposition}{\rm (\cite[Corollary~7.22]{Vi2})}\label{pr:LV}
Let $(X,d)$ be a locally compact geodesic space.
Then, for any $\mu,\nu \in \cP^2(X)$ and any minimal geodesic
$\alpha:[0,1] \lra \cP^2(X)$ between them, there exists $\Pi \in \cP(\Gamma(X))$
such that $(e_0 \times e_1)_{\sharp}\Pi$ is an optimal coupling of $\mu$ and $\nu$
and that $(e_t)_{\sharp}\Pi=\alpha(t)$ holds for all $t \in [0,1]$.
\end{proposition}

We denoted by $(e_t)_{\sharp}\Pi$ the push-forward measure of $\Pi$ by $e_t$.
In Riemannian manifolds, a more precise description of an optimal transport
using a gradient vector field of some kind of convex function is known.
We first recall McCann's original work on compact Riemannian manifolds.
Denote by $\cP_{\ac}(M,\vol_g) \subset \cP(M)$ the subset of absolutely continuous
measures with respect to the volume measure $\vol_g$.
We also set $\cP_{\ac}^2(M,\vol_g):=\cP^2(M) \cap \cP_{\ac}(M,\vol_g)$.

\begin{theorem}{\rm (\cite[Theorems~8, 9]{Mc2})}\label{th:Mc}
Let $(M,g)$ be a compact Riemannian manifold.
Then, for any $\mu \in \cP_{\ac}(M,\vol_g)$ and $\nu \in \cP(M)$,
there exists a $(d^2/2)$-convex function $\varphi:M \lra \R$ such that the map
$\cT_t(x):=\exp_x(t\nabla\varphi(x))$, $t \in [0,1]$, provides a unique
minimal geodesic from $\mu$ to $\nu$.
Precisely, $(\cT_0 \times \cT_1)_{\sharp}\mu$ is an optimal coupling
of $\mu$ and $\nu$, and $\mu_t=(\cT_t)_{\sharp}\mu$ is a minimal geodesic
from $\mu_0=\mu$ to $\mu_1=\nu$ with respect to $W_2$.
\end{theorem}

See \cite[Chapter~5]{Vi2} for the definition of the {\it $(d^2/2)$-convex function},
here we only remark that it is semi-convex in compact spaces.
Such convexity is important as it implies the almost everywhere
twice differentiability (due to the Alexandrov-Bangert theorem),
and is generalized to noncompact spaces in \cite{FG}.

\begin{theorem}{\rm (\cite[Theorem~1]{FG})}\label{th:FG}
Let $(M,g)$ be a complete Riemannian manifold.
Then, for any $\mu \in \cP^2_{\ac}(M,\vol_g)$ and $\nu \in \cP^2(M)$,
there exists a locally semi-convex function $\varphi:\Omega \lra \R$
on an open set $\Omega \subset M$ with $\mu(\Omega)=1$
such that the map $\cT_t(x):=\exp_x(t\nabla\varphi(x))$,
$t \in [0,1]$, provides a unique minimal geodesic from $\mu$ to $\nu$
$($in the sense of Theorem~$\ref{th:Mc})$.
\end{theorem}

We will also use the following {\it Jacobian} (or {\it Monge-Amper\`e}) {\it equation}.

\begin{theorem}{\rm (\cite[Theorems~8.7, 11.1]{Vi2})}\label{th:MA}
Let $(M,g)$ be complete and $\mu$, $\nu$, $\varphi$, $\Omega$ and $\cT_t$
be as in Theorem~$\ref{th:FG}$ above.
Put
\[ \bJ_t^{\omega}(x):=e^{\psi(x)-\psi(\cT_t(x))} \det\big( D\cT_t(x) \big) \]
for $x \in \Omega$ and $t \in [0,1)$.
Then it holds $\mu_t \in \cP^2_{\ac}(M,\vol_g)$ and
$(\rho_t \circ \cT_t) \bJ^{\omega}_t =\rho_0$ $\mu_0$-a.e.\
for all $t \in [0,1)$, where we set $\mu_t=(\cT_t)_{\sharp}\mu=\rho_t \omega$.
In particular, $\bJ^{\omega}_t >0$ $\mu_0$-a.e.\ for each $t \in [0,1)$.
If in addition $\nu \in \cP_{\ac}^2(M,\vol_g)$, then the above assertions
hold also at $t=1$.
\end{theorem}

Note that $\bJ^{\omega}_t$ is the combination of the Jacobian $\det(D\cT_t)$
of $\cT_t$ with respect to the metric $g$ and the ratio $e^{\psi-\psi(\cT_t)}$
of the weight $e^{-\psi}$ on $\vol_g$.

\section{Generalized relative entropies}\label{sc:Hm}

Before discussing the $m$-relative entropy, we briefly review
the Boltzmann and the Tsallis entropies (see \cite{Ts1}, \cite{Ts2}),
and explain the motivation related to information geometry
(see \cite{Am}, \cite{AN}).

\subsection{Background: Tsallis entropy and information geometry}\label{ssc:ent}

Entropy is a functional playing prominent roles in thermodynamics,
information theory (sometimes with the opposite sign) and many other fields.
It describes how particles diffuse in thermodynamics, and measures
the uncertainty of an event in information theory.
The most fundamental entropy is the {\it Boltzmann$($-Gibbs-Shannon$)$ entropy}
given by
\[ E(\mu)=-\int_{\R^n} \rho \ln\rho \,dx \]
for $\mu=\rho dx \in \cP_{\ac}(\R^n,dx)$, where $dx$ is the Lebesgue measure.

Boltzmann entropy is thermodynamically {\it extensive} and probabilistically {\it additive},
so that it is suitable for the treatment of independent systems.
Precisely, for two independent distributions $\mu_1,\mu_2 \in \cP_{\ac}(\R^n,dx)$
and their joint probability $\mu_1 \times \mu_2 \in \cP_{\ac}(\R^{2n},dx)$,
one easily observes $E(\mu_1 \times \mu_2)=E(\mu_1)+E(\mu_2)$.
Recently, there is a growing interest in strongly correlated systems
and non-additive entropies.
Among them, we are interested in the {\it Tsallis entropy} defined by
\begin{equation}\label{eq:Em}
E_m(\mu):=-\int_{\R^n} e_m(\rho) \,dx =-\int_{\R^n} \rho \ln_m \rho \,dx
 =-\int_{\R^n} \frac{\rho^m-\rho}{m-1} \,dx
\end{equation}
for $\mu=\rho dx \in \cP_{\ac}(\R^n,dx)$, where $m \in [(n-1)/n,1) \cup (1,2]$.
Note that letting $m$ tend to $1$ recovers the Boltzmann entropy $E(\mu)$,
and that $E_m(\mu)$ is closely related to the {\it R\'enyi entropy}
\begin{equation}\label{eq:Ren}
S_N(\mu):=-\int_{\R^n} \rho^{1-1/N} \,dx =(m-1)E_m(\mu)-1.
\end{equation}
One can connect $E$ and $E_m$ via Gaussian measures as follows.
On the one hand, given $v \in \R^n$ and $V \in \Sym^+(n,\R)$,
the (usual) Gaussian measure
\[ N(v,V)=\frac{1}{(2\pi)^{n/2} (\det V)^{1/2}}
 \exp\bigg[ -\frac{1}{2}\langle x-v,V^{-1}(x-v) \rangle \bigg] dx \]
maximizes $E$ among $\mu \in \cP_{\ac}(\R^n,dx)$ with mean $v$
and covariance matrix $V$.
On the other hand, the $m$-Gaussian measure $N_m(v,V)$ defined in $(\ref{eq:Nm})$
similarly maximizes $E_{2-m}$ under the same constraint (for $m \neq 1/2,2$).

In the following sections, we shall verify that a number of further
geometric and analytic properties of $E$ have counterparts for $E_m$.
Precisely, since $E_m$ itself is not really interesting in our view
(see Remark~\ref{rm:N<n}(2)), we modify $E_m$ in the manner of information geometry.

We start from the family of Gaussian measures
\[ \cN(n):=\{ N(v,V) \,|\, v \in \R^n,\, V \in \Sym^+(n,\R) \} \]
as an $((n^2+3n)/2)$-dimensional manifold.
In information geometry, we equip $\cN(n)$ with the {\it Fisher information metric}
$m_F$ which is different from the Wasserstein metric $W_2$.
In fact, $(\cN(1),m_F)$ has the negative constant sectional curvature (\cite{Am}),
while $(\cN(1),W_2)$ is flat (cf.\ \cite[Theorem~2.2]{Ta1} and the references therein).
The Fisher metric admits a pair of dually flat connections
({\it exponential} and {\it mixture connections}) and the {\it Kullback-Leibler divergence}
\[ H(\mu|\nu) =\int_{\R^n}
 \frac{\rho}{\sigma} \ln\bigg( \frac{\rho}{\sigma} \bigg) \,d\nu \]
for $\nu=\sigma dx \in \cP_{\ac}(\R^n,dx)$ and $\mu=\rho dx \in \cP_{\ac}(\R^n,\nu)$.
Note that $H(\mu|\nu)$ is nonnegative by Jensen's inequality.
The square root of the divergence $H(\mu|\nu)$ can be regarded as a kind of
distance between $\mu$ and $\nu$.
It certainly satisfies a generalized Pythagorean theorem,
though it does not satisfy symmetry nor the triangle inequality.
The Kullback-Leibler divergence $H(\mu|\nu)$ coincides with the
{\it relative entropy} $\Ent_{\nu}(\mu)$ of $\mu$ with respect to $\nu$.
Roughly speaking, $\Ent_{\nu}(\mu)$ is defined for $\mu \in \cP(\R^n)$
and a Borel measure $\nu$ on $\R^n$ by
\begin{equation}\label{eq:rent}
\Ent_{\nu}(\mu):= \left\{ \begin{array}{cl}
 \int_{\R^n} \varsigma \ln \varsigma \,d\nu
 & {\rm for}\ \mu=\varsigma \nu \in \cP_{\ac}(\R^n,\nu), \\
 \infty & {\rm otherwise},
\end{array} \right.
\end{equation}
and then $\Ent_{\nu}(\mu) \ge -\ln\nu(\R^n)$.

The family of $m$-Gaussian measures
\[ \cN(n,m):=\{ N_m(v,V) \,|\, v \in \R^n,\, V \in \Sym^+(n,\R) \} \]
similarly admits dually flat connections and the corresponding {\it Bregman divergence}
(called the {\it $\beta$-divergence}, cf.\ \cite[\S 2.1]{OW}) is
\begin{equation}\label{eq:Breg}
H_m(\mu|\nu)=\frac{1}{m(m-1)} \int_{\R^n}
\{ \rho^m-m\rho \sigma^{m-1}+(m-1)\sigma^m \} \,dx
\end{equation}
for $\nu=\sigma dx \in \cP_{\ac}(\R^n,dx)$ and $\mu=\rho dx \in \cP_{\ac}(\R^n,\nu)$.
We can rewrite this by using $e_m$ as
\[ H_m(\mu|\nu)=\frac{1}{m} \int_{\R^n}
\{ e_m(\rho) -e_m(\sigma) -e'_m(\sigma)(\rho-\sigma) \} \,dx \]
and recover the Kullback-Leibler divergence as the limit:
\[ \lim_{m \to 1}H_m(\mu|\nu) =\int_{\R^n}
 \{ \rho\ln\rho -\sigma\ln\sigma -(\ln\sigma+1)(\rho-\sigma) \} \,dx
 =H(\mu|\nu). \]
It will turn out that the entropy induced from $(\ref{eq:Breg})$ is appropriate
for our purpose.
We remark that the division by $m$ in $(\ref{eq:Breg})$ is unessential,
we prefer this form merely for aesthetic reasons of the presentation of Theorem~\ref{th:mCD}.

\subsection{$m$-relative entropy}\label{ssc:Hm}

Recall our weighted Riemannian manifold $(M,\omega)$ and reference measure
$\nu=\sigma \omega$.
The Bregman divergence $(\ref{eq:Breg})$ leads us to the following
generalization of the relative entropy.

\begin{definition}[$m$-relative entropy]\label{df:Hm}
Assume $\sigma \in L^m(M,\omega)$.
Given $\mu \in \cP(M)$, let $\mu=\rho\omega +\mu^s$ be its
Lebesgue decomposition into absolutely continuous and singular
parts with respect to $\omega$.
Then we define the {\it $m$-relative entropy} as follows.

(1) For $m<1$,
\begin{align}
&H_m(\mu|\nu) \nonumber\\
&:= \frac{1}{m}\int_M
 \{ e_m(\rho)-e_m(\sigma)-e'_m(\sigma)(\rho-\sigma) \} \,d\omega
 -\frac{1}{m-1}\int_M \sigma^{m-1} \,d\mu^s +H_m(\infty) \mu^s(M)  \nonumber\\
&= \frac{1}{m(m-1)}\int_M \{ \rho^m+(m-1)\sigma^m \} \,d\omega
 -\frac{1}{m-1}\int_M \sigma^{m-1} \,d\mu
 +H_m(\infty) \mu^s(M) \label{eq:Hm}
\end{align}
if $\sigma \in L^{m-1}(M,\mu)$, where $H_m(\infty):=0$.
We define $H_m(\mu|\nu):=\infty$ for $\mu \in \cP(M)$ with $\sigma \not\in L^{m-1}(M,\mu)$.

(2) For $m>1$, $H_m(\mu|\nu)$ is defined by \eqref{eq:Hm} if $\rho \in L^m(M,\omega)$,
where $H_m(\infty):=\infty$ and $\infty \cdot 0=0$ as convention.
We set $H_m(\mu|\nu):=\infty$ for $\mu \in \cP(M)$ with $\rho \not\in L^m(M,\omega)$.
\end{definition}

For $\mu=\rho\omega \in \cP_{\ac}(M,\omega)$, \eqref{eq:Hm} has the simplified form
\[ H_m(\mu|\nu)= \frac{1}{m(m-1)}\int_M
 \{ \rho^m-m\rho\sigma^{m-1}+(m-1)\sigma^m \} \,d\omega \]
as in $(\ref{eq:Breg})$.
Note that the first two terms in the right hand side are regarded as the internal
and external energies, and the last term (which is independent of $\mu$)
is added for the sake of nonnegativity (see Lemma~\ref{lm:Hm}).

\begin{remark}\label{rm:Hm}
(1) If $\Hess\Psi \ge K>0$, then the primal assumption $\sigma \in L^m(M,\omega)$
is clearly satisfied for $m>1$ by Lemma~\ref{lm:K>0}(iii).
We deduce from Lemma~\ref{lm:K>0}(i) that $\sigma \in L^m(M,\omega)$
also holds true if $\Hess\Psi \ge K>0$, $\Ric_N \ge 0$ and $m \in (1/2,1)$.

(2-a) For $m<1$, if $\sigma \in L^{m-1}(M,\mu)$, then the H\"older inequality implies
\[ \int_M \rho^m \,d\omega
 =\int_M (\rho \sigma^{m-1})^m \sigma^{m(1-m)} \,d\omega
 \le \bigg( \int_M \rho \sigma^{m-1} \,d\omega \bigg)^m
 \bigg( \int_M \sigma^m \,d\omega \bigg)^{1-m}. \]
Thus we have $\rho \in L^m(M,\omega)$.
Moreover, for $\mu=\rho\omega \in \cP_{\ac}(M,\omega)$, it holds
\begin{align*}
&H_m(\mu|\nu)  -\frac{1}{m} \int_M \sigma^m \,d\omega \\
&\ge \frac{1}{m(m-1)} \bigg( \int_M \sigma^{m-1} \,d\mu \bigg)^m
 \bigg( \int_M \sigma^m \,d\omega \bigg)^{1-m}
 +\frac{1}{1-m} \int_M \sigma^{m-1} \,d\mu \\
&= \frac{1}{m(1-m)} \bigg( \int_M \sigma^{m-1} \,d\mu \bigg)^m
 \bigg\{ m\bigg( \int_M \sigma^{m-1} \,d\mu \bigg)^{1-m}
 -\bigg( \int_M \sigma^m \,d\omega \bigg)^{1-m} \bigg\},
\end{align*}
and hence it is natural to define $H_m(\mu|\nu)=\infty$ for
$\mu$ with $\sigma \not\in L^{m-1}(M,\mu)$.

(2-b) For $m>1$ and $\rho \in L^m(M,\omega)$, the H\"older inequality
\[ \int_M \rho \sigma^{m-1} \,d\omega \le
 \bigg( \int_M \rho^m \,d\omega \bigg)^{1/m}
 \bigg( \int_M \sigma^m \,d\omega \bigg)^{(m-1)/m} \]
similarly yields $\sigma \in L^{m-1}(M,\mu)$ and,
for $\mu=\rho\omega \in \cP_{\ac}(M,\omega)$,
\begin{align*}
&H_m(\mu|\nu)  -\frac{1}{m} \int_M \sigma^m \,d\omega \\
&\ge \frac{1}{m(m-1)} \bigg( \int_M \rho^m \,d\omega \bigg)^{1/m}
 \bigg\{ \bigg( \int_M \rho^m \,d\omega \bigg)^{(m-1)/m}
 -m\bigg( \int_M \sigma^m \,d\omega \bigg)^{(m-1)/m} \bigg\}.
\end{align*}
Hence it is again natural to set $H_m(\mu|\nu)=\infty$ for
$\rho \not\in L^m(M,\omega)$.

(3) The validity of the definition of $H_m(\infty)$ would be understood
by the following observation (putting $\rho=\chi_{B(x,\ve)}/\omega(B(x,\ve))$
so that $\chi_{B(x,\ve)}$ is the characteristic function of $B(x,\ve)$):
\[ \int_{B(x,\ve)} \frac{1}{\omega(B(x,\ve))^m} \,d\omega
 =\omega\big( B(x,\ve) \big)^{1-m} \ \to \
 \left\{ \begin{array}{cl} 0 & {\rm if}\ m<1, \\
 \infty & {\rm if}\ m>1 \end{array} \right. \]
as $\ve$ tends to zero (see also Lemma~\ref{lm:lsc} below).
\end{remark}

Next we see that $\nu$ is a unique {\it ground state} of $H_m(\cdot|\nu)$
(provided $\nu(M)=1$).

\begin{lemma}\label{lm:Hm}
We have $H_m(\mu|\nu) \ge 0$ for all $\mu \in \cP(M)$,
and equality holds if and only if $\nu \in \cP_{\ac}(M,\omega)$ and $\mu=\nu$.
\end{lemma}

\begin{proof}
Note that, if $\mu^s(M)>0$, then the singular part
\[ -\frac{1}{m-1}\int_M \sigma^{m-1} \,d\mu^s +H_m(\infty) \mu^s(M) \]
in $(\ref{eq:Hm})$ is positive for $m<1$ (since $\sigma>0$ on $M$)
and infinity for $m>1$, respectively.
Hence it is sufficient to consider the absolutely continuous part.
As the function $e_m(t)=(t^m-t)/(m-1)$ is strictly convex on $(0,\infty)$,
we have
\[ e_m(\rho)-e_m(\sigma)-e'_m(\sigma)(\rho-\sigma) \ge 0 \]
in $(\ref{eq:Hm})$ and equality holds if and only if $\rho=\sigma$.
Therefore $H_m(\mu|\nu) \ge 0$ and equality holds if and only if
$\mu^s(M)=0$ and $\rho=\sigma$ $\omega$-a.e..
$\qedd$
\end{proof}

The following lemma will be used in Section~\ref{sc:gf} (Claim~\ref{cl:gf})
where $M$ is assumed to be compact.
This also guarantees the validity of the definition of $H_m(\infty)$.

\begin{lemma}\label{lm:lsc}
Let $(M,g)$ be compact.
Then the entropy $H_m(\cdot|\nu)$ is lower semi-continuous with respect to
the weak topology, that is to say, if a sequence $\{ \mu_i \}_{i \in \N} \subset \cP(M)$
weakly converges to $\mu \in \cP(M)$, then we have
\[ H_m(\mu|\nu) \le \liminf_{i \to \infty} H_m(\mu_i|\nu). \]
\end{lemma}

\begin{proof}
We divide $H_m(\mu|\nu)-m^{-1}\int_M \sigma^m \,d\omega$ into two parts:
\[ h_1(\mu):=\frac{1}{m(m-1)} \int_M \rho^m \,d\omega +H_m(\infty)\mu^s(M),
 \quad h_2(\mu):=-\frac{1}{m-1} \int_M \sigma^{m-1} \,d\mu. \]
Then $h_2(\mu)$ is clearly continuous in $\mu$ (since $M$ is compact).
In addition, the lower semi-continuity of $h_1(\mu)$ follows from
\cite[Theorem~B.33]{LV2} since the function $U_m(t):=t^m/m(m-1)$ is continuous,
convex and satisfies $U_m(0)=0$ as well as $\lim_{t \to \infty}U_m(t)/t=H_m(\infty)$.
$\qedd$
\end{proof}

\section{Displacement convexity}\label{sc:dc}

In this section, we prove our first main theorem on a characterization of
the displacement convexity of $H_m(\cdot|\nu)$ along the lines of
\cite{CMS}, \cite{vRS}, \cite{Stcon} and \cite{StII}.

In \cite{Stcon}, Sturm considered a more general class of entropies
(or free energies) on unweighted Riemannian manifolds.
Then his \cite[Theorem~1.3]{Stcon} includes the equivalence between (A) and (B)
in Theorem~\ref{th:mCD} below (with $\omega=\vol_g$, see also \cite[Remark~1.1]{Stcon}).
To be precise, in his theorem, the condition (A) is written as
\[ U'(r)\Ric(v) +\Hess\Psi(v,v) \ge K \]
for all $r \in \R$ and unit vectors $v \in TM$, where $U(r)=e^{(m-1)r}/m(m-1)$
(one more condition $U''(r)+U'(r)/n \ge 0$ corresponds to $m \ge (n-1)/n$,
see Remark~\ref{rm:N<n}(1)).
Thus Theorem~\ref{th:mCD} can be regarded as the combination of
\cite[Theorem~1.3]{Stcon} and the equivalence between $\Ric_N \ge K$ and $\CD(K,N)$
(for $(M,\omega)$, see \cite[Theorem~1.7]{StII}, \cite[Theorem~4.22]{LV1}).
Our proof is also in a sense the combination of them.
Recall from \eqref{eq:M0} that $M_0=M$ for $m<1$, $M_0=\Psi^{-1}((-\infty,1/(m-1)))$
for $m>1$, and that $\overline{M_0}=\supp\nu$ in both cases.

\begin{theorem}\label{th:mCD}
Let $(M,\omega,\nu)$ and $m \in [(n-1)/n,1) \cup (1,\infty)$ with
$\sigma \in L^m(M,\omega)$ be given.
Then, for $K \in \R$, the following three conditions are mutually equivalent$:$
\begin{enumerate}[{\rm (A)}]
\item We have $\Ric_N \ge 0$ on $\overline{M_0}$ with $N=1/(1-m)$ as well as
$\Hess\Psi \ge K$ on $\overline{M_0}$ in the sense of Definition~$\ref{df:Psi}$.

\item For any $\mu_0,\mu_1 \in \cP^2_{\ac}(\overline{M_0},\omega)$ such that
any two points $x_0 \in \supp\mu_0$, $x_1 \in \supp\mu_1$ are joined by
some geodesic contained in $\overline{M_0}$, there is a minimal geodesic
$(\mu_t)_{t \in [0,1]} \subset \cP^2_{\ac}(\overline{M_0},\omega)$ along which we have
\begin{equation}\label{eq:mCD}
H_m(\mu_t|\nu) \le (1-t)H_m(\mu_0|\nu) +tH_m(\mu_1|\nu)
 -\frac{K}{2}(1-t)tW_2(\mu_0,\mu_1)^2
\end{equation}
for all $t \in [0,1]$.

\item For any $\mu_0,\mu_1 \in \cP^2(\overline{M_0})$ such that
any two points $x_0 \in \supp\mu_0$, $x_1 \in \supp\mu_1$ are joined by
some geodesic contained in $\overline{M_0}$, there is a minimal geodesic
$(\mu_t)_{t \in [0,1]} \subset \cP^2(\overline{M_0})$ along which we have
$(\ref{eq:mCD})$ for all $t \in [0,1]$.
\end{enumerate}
\end{theorem}

\begin{proof}
Note that (C) $\Rightarrow$ (B) is clear.
Thus it suffices to show (A) $\Rightarrow$ (C) and (B) $\Rightarrow$ (A).
As the general case of the part (A) $\Rightarrow$ (C) is somewhat technical,
let us begin with absolutely continuous measures, in other words,
(A) $\Rightarrow$ (B).

(A) $\Rightarrow$ (B):
Since the assertion $(\ref{eq:mCD})$ is clear
if $H_m(\mu_0|\nu)=\infty$ or $H_m(\mu_1|\nu)=\infty$, we assume
that both $H_m(\mu_0|\nu)$ and $H_m(\mu_1|\nu)$ are finite.
Theorem~\ref{th:FG} ensures that there is an almost everywhere
twice differentiable function $\varphi:M \lra \R$ such that the map
$\cT_t(x):=\exp_x(t\nabla\varphi(x))$ gives the unique minimal geodesic
$\mu_t:=(\cT_t)_{\sharp}\mu_0$ from $\mu_0$ to $\mu_1$.
Due to \cite[Proposition~4.1]{CMS}, $\cT_1(x)$ is not a cut point of $x$
for $\mu_0$-a.e.\ $x$, and hence the minimal geodesic $(\cT_t(x))_{t \in [0,1]}$
is unique and contained in $\overline{M_0}$.
Recall that, putting $\mu_t=\rho_t \omega$,
\[ H_m(\mu_t|\nu)=\frac{1}{m(m-1)}
 \int_M (\rho_t^{m-1}-m\sigma^{m-1}) \,d\mu_t
 +\frac{1}{m}\int_M \sigma^m \,d\omega. \]
By the Jacobian equation (Theorem~\ref{th:MA}), we deduce that
\begin{align*}
\int_M (\rho_t^{m-1}-m\sigma^{m-1}) \,d\mu_t
&= \int_M \{ \rho_t(\cT_t)^{m-1}-m\sigma(\cT_t)^{m-1} \} \,d\mu_0 \\
&= \int_M \bigg\{ \bigg( \frac{\bJ^{\omega}_t}{\rho_0} \bigg)^{1-m}
 -m\sigma(\cT_t)^{m-1} \bigg\} \,d\mu_0,
\end{align*}
where $\bJ_t^{\omega}(x):=e^{\psi(x)-\psi(\cT_t(x))} \det(D\cT_t(x))>0$ $\mu_0$-a.e..

\begin{claim}\label{cl:Jt}
For $\mu_0$-a.e.\ $x \in M$, the function
$\bJ^{\omega}_t(x)^{1-m}/(m-1)=-N \bJ^{\omega}_t(x)^{1/N}$ is convex in $t$.
\end{claim}

\begin{proof}
For $m<1$ (and hence $N \ge n$), this is proved in \cite[Theorem~1.7]{StII}
(see also \cite[Section~8.2]{Oint}).
We can apply the same calculation to $m>1$ (and $N<0$).
For completeness, we briefly explain how to modify calculations in \cite{Oint}.
With the notations in \cite[Section~8.2]{Oint},
we observe that $\Ric_N \ge 0$ implies $(N-1)h''_3 h_3^{-1} \le 0$.
Thus $h_3$ is convex and $e^{\beta}$ is concave, therefore
\[ \{ e^{-\psi(x)}\bJ^{\omega}_t(x) \}^{1/N}=h(t)
 =(e^{\beta(t)})^{1/N} h_3(t)^{(N-1)/N} \]
is convex in $t$ (via the H\"older inequality
\[ (a+b)^{1/N}(c+d)^{(N-1)/N} \le a^{1/N}c^{(N-1)/N}+b^{1/N}d^{(N-1)/N} \]
for $a,b >0$ and $c,d \ge 0$).
$\hfill \diamondsuit$
\end{proof}

In order to estimate the term $\sigma(\cT_t)^{m-1}/(1-m)$,
we observe from $\Hess\Psi \ge K$ that
\begin{align*}
\frac{\sigma(\cT_t)^{m-1}}{1-m}
&=\frac{1}{1-m}+\Psi(\cT_t) \\
&\le \frac{1}{1-m}+(1-t)\Psi(\cT_0)+t\Psi(\cT_1)
 -\frac{K}{2}(1-t)td(\cT_0,\cT_1)^2 \\
&= (1-t)\frac{\sigma(\cT_0)^{m-1}}{1-m} +t\frac{\sigma(\cT_1)^{m-1}}{1-m}
 -\frac{K}{2}(1-t)td(\cT_0,\cT_1)^2.
\end{align*}
Combining this with Claim~\ref{cl:Jt} and integrating with $\mu_0$
yield the desired inequality $(\ref{eq:mCD})$.

(A) $\Rightarrow$ (C):
We next consider the more technical case where
$\mu_0$ or $\mu_1$ has nontrivial singular part.
There is nothing to prove for $m>1$.
For $m<1$, we decompose as $\mu_0=\rho_0 \omega+\mu_0^s$
and $\mu_1=\rho_1 \omega+\mu_1^s$, and take an optimal coupling
$\pi$ of $\mu_0$ and $\mu_1$.
Now, $\pi$ is decomposed into four parts
$\pi=\pi_{aa}+\pi_{as}+\pi_{sa}+\pi_{ss}$ such that
$(p_1)_{\sharp}(\pi_{aa})$, $(p_1)_{\sharp}(\pi_{as})$, $(p_2)_{\sharp}(\pi_{aa})$
and $(p_2)_{\sharp}(\pi_{sa})$ are absolutely continuous,
and that $(p_1)_{\sharp}(\pi_{sa})$, $(p_1)_{\sharp}(\pi_{ss})$,
$(p_2)_{\sharp}(\pi_{as})$ and $(p_2)_{\sharp}(\pi_{ss})$ are singular
(or null) measures.
Here $p_1,p_2:M \times M \lra M$ denote projections to the first and second elements.

We divide optimal transport between $\mu_0$ and $\mu_1$ into two parts,
corresponding to $\pi-\pi_{ss}$ and $\pi_{ss}$.
As for $\hat{\mu}_0:=(p_1)_{\sharp}(\pi-\pi_{ss})$ and
$\hat{\mu}_1:=(p_2)_{\sharp}(\pi-\pi_{ss})$, Theorems~\ref{th:FG}, \ref{th:MA} are again
applicable and give a minimal geodesic $\hat{\mu}_t=\hat{\rho}_t \omega
 \in (1-\pi_{ss}(M \times M)) \cdot \cP^2_{\ac}(\overline{M_0},\omega)$
(i.e., $\hat{\mu}_t(M)=1-\pi_{ss}(M \times M)$) satisfying
\begin{align*}
\int_M \hat{\rho}_t^m \,d\omega
&\ge (1-t)\int_M \rho_0^m \,d\omega +t\int_M \rho_1^m \,d\omega, \\
\int_M \sigma^{m-1} \,d\hat{\mu}_t
&\le (1-t)\int_M \sigma^{m-1} \,d\hat{\mu}_0
 +t\int_M \sigma^{m-1} \,d\hat{\mu}_1 \\
&\quad -\frac{(1-m)K}{2}(1-t)t \int_{M \times M} d(x,y)^2 \,d(\pi-\pi_{ss})(x,y).
\end{align*}
We then choose an arbitrary minimal geodesic
$\tilde{\mu}_t=\tilde{\rho}_t \omega +\tilde{\mu}_t^s
 \in \pi_{ss}(M \times M) \cdot \cP^2(\overline{M_0})$ from
$\tilde{\mu}_0:=(p_1)_{\sharp}(\pi_{ss})$ to
$\tilde{\mu}_1:=(p_2)_{\sharp}(\pi_{ss})$.
Thanks to Proposition~\ref{pr:LV}, $\tilde{\mu}_t$ is also realized through
a family of geodesics in $\overline{M_0}$, and hence $\Hess\Psi \ge K$ implies
\begin{align*}
\int_M \sigma^{m-1} \,d\tilde{\mu}_t
&\le (1-t)\int_M \sigma^{m-1} \,d\tilde{\mu}_0
 +t\int_M \sigma^{m-1} \,d\tilde{\mu}_1 \\
&\quad -\frac{(1-m)K}{2}(1-t)t \int_{M \times M} d(x,y)^2 \,d\pi_{ss}(x,y).
\end{align*}
We put $\mu_t:=\hat{\mu}_t+\tilde{\mu}_t$ and conclude that
\begin{align*}
H_m(\mu_t|\nu) &= \frac{1}{m(m-1)}
 \int_M \{ (\hat{\rho}_t+\tilde{\rho}_t)^m +(m-1)\sigma^m \} \,d\omega
 +\frac{1}{1-m} \int_M \sigma^{m-1} \,d\mu_t \\
&\le \frac{1}{m(m-1)} \int_M \{ \hat{\rho}_t^m +(m-1)\sigma^m \} \,d\omega
 +\frac{1}{1-m} \int_M \sigma^{m-1} \,d(\hat{\mu}_t+\tilde{\mu}_t) \\
&\le (1-t)H_m(\mu_0|\nu) +tH_m(\mu_1|\nu) -\frac{K}{2}(1-t)tW_2(\mu_0,\mu_1)^2.
\end{align*}

(B) $\Rightarrow$ (A):
By approximation, it suffices to show $\Ric_N \ge 0$ and $\Hess\Psi \ge K$ on $M_0$.
We first consider the case of $m<1$.
Fix a unit vector $v \in T_xM$ with $x \in M_0$ and put
$\gamma(t):=\exp_x(tv)$, $B_{\pm}:=B(\gamma(\pm\delta),(1\mp a\delta)\ve)$
for $0<\ve \ll \delta \ll 1$ with a constant $a \in \R$ chosen later.
Set
\begin{equation}\label{eq:mu0}
\mu_0=\rho_0 \omega:=\frac{\chi_{B_-}}{\omega(B_-)}\omega,
  \quad \mu_1=\rho_1 \omega:=\frac{\chi_{B_+}}{\omega(B_+)}\omega,
\end{equation}
where $\chi_A$ stands for the characteristic
function of a set $A$.
Let $\mu_t=(\cT_t)_{\sharp}\mu_0$ be the unique optimal transport from $\mu_0$
to $\mu_1$.
Recall that
\begin{equation}\label{eq:BA}
H_m(\mu_t|\nu) -\frac{1}{m}\int_M \sigma^m\,d\omega
 =\frac{1}{m(m-1)} \int_M \{ \rho_0^{m-1} (\bJ^{\omega}_t)^{1-m}
 -m\sigma(\cT_t)^{m-1} \} \,d\mu_0,
\end{equation}
where $\bJ^{\omega}_t=e^{\psi-\psi(\cT_t)}\det(D\cT_t)$.
By definition, we find
\[ \rho_0^{m-1}
 =\{ c_n e^{-\psi(\gamma(-\delta))}(1+a\delta)^n\ve^n +O(\ve^{n+1}) \}^{1-m}
 \chi_{B_-}, \]
where $c_n$ denotes the volume of the unit ball in $\R^n$.
Note also that
\begin{equation}\label{eq:DC+}
\int_M (\bJ^{\omega}_t)^{1-m} \,d\mu_0
 \le \bigg( \int_M \bJ^{\omega}_t \,d\mu_0 \bigg)^{1-m}
 =\bigg( \frac{\omega(\supp \mu_t)}{\omega(B_-)} \bigg)^{1-m}.
\end{equation}
As the (second order) behavior of the distance function is controlled
by the sectional curvature, we have
\begin{align*}
\supp\mu_{1/2} &\subset \exp_x \bigg(
 \sum_{i=1}^n a_i \frac{\del}{\del x^i} \in T_xM \,\bigg|\,
 \sqrt{\sum_{i=1}^n \bigg( \frac{a_i}{\ve_i} \bigg)^2} \le 1 \bigg), \\
\ve_i &:=\bigg( 1+\frac{k_i}{2}\delta^2 +O(\delta^3) \bigg) \ve,
\end{align*}
where we chose a coordinate $(x^i)_{i=1}^n$ around $x$ such that
$\{ (\del/\del x^i)|_x \}_{i=1}^n$ is orthonormal and that
$(\del/\del x^1)|_x=\dot{\gamma}(0)$,
and denote by $k_i$ the sectional curvature of the plane spanned by
$\dot{\gamma}(0)$ and $(\del/\del x^i)|_x$ (so that $k_1=0$)
(see the proof of \cite[Theorem~1]{vRS}).
Thus we observe from $\Ric(v)=\sum_{i=1}^{n}k_i$ that
\begin{align}
\limsup_{\ve \to 0} \frac{\omega(\supp \mu_{1/2})}{c_n \ve^n}
&=e^{-\psi(x)}\limsup_{\ve \to 0} \frac{\vol_g(\supp \mu_{1/2})}{c_n \ve^n} \nonumber\\
&\le e^{-\psi(x)} \bigg\{ 1+\frac{1}{2}\Ric(v)\delta^2 +O(\delta^3) \bigg\}. \label{eq:DC+2}
\end{align}
We similarly observe that $\omega(\supp \mu_t)/c_n \ve^n$
is uniformly bounded as $\ve \to 0$.
Hence, since $1-m>0$, the leading term of $(\ref{eq:BA})$ (as $\ve \to 0$) is
\[ \frac{1}{1-m}\int_M \sigma(\cT_t)^{m-1} \,d\mu_0. \]
Thus we obtain from $(\ref{eq:mCD})$ with $t=1/2$ that, by letting $\ve$ go to zero,
\[ \sigma\big( \gamma(0) \big)^{m-1}
 \le \frac{\sigma(\gamma(-\delta))^{m-1}+\sigma(\gamma(\delta))^{m-1}}{2}
 -(1-m)\frac{K}{8}(2\delta)^2. \]
This means that
\[ \Hess\Psi =\frac{1}{1-m}\Hess(\sigma^{m-1}) \ge K \]
in the weak sense.

In order to show $\Ric_N(v) \ge 0$, we choose a point $y$ with $d(x,y) \gg \delta$
and modify $\mu_0$ and $\mu_1$ into
\begin{equation}\label{eq:mu1}
\tilde{\mu}_i:=
 (1-\ve^{n+1})\frac{\chi_{B(y,\delta)}}{\omega(B(y,\delta))}\omega
 +\ve^{n+1}\mu_i
\end{equation}
for $i=0,1$.
Then
$W_2(\tilde{\mu}_0,\tilde{\mu}_1)=\ve^{(n+1)/2} \cdot W_2(\mu_0,\mu_1)$
and
\[ \tilde{\mu}_t:=
 (1-\ve^{n+1})\frac{\chi_{B(y,\delta)}}{\omega(B(y,\delta))}\omega
 +\ve^{n+1}\mu_t \]
is the unique minimal geodesic from $\tilde{\mu}_0$ to $\tilde{\mu}_1$,
so that $(\ref{eq:BA})$ is modified into
\begin{align*}
&H_m(\tilde{\mu}_t|\nu) -\frac{1}{m}\int_M \sigma^m \,d\omega \\
&= \frac{\ve^{n+1}}{m(m-1)}
 \int_M \{ (\ve^{n+1}\rho_0)^{m-1}(\bJ^{\omega}_t)^{1-m}
 -m\sigma(\cT_t)^{m-1} \} \,d\mu_0 \\
&\quad +\frac{1}{m(m-1)} \frac{1-\ve^{n+1}}{\omega(B(y,\delta))} \int_{B(y,\delta)}
 \bigg\{ \bigg( \frac{1-\ve^{n+1}}{\omega(B(y,\delta))} \bigg)^{m-1}
 -m\sigma^{m-1} \bigg\} \,d\omega.
\end{align*}
We rewrite this as
\begin{align}
&H_m(\tilde{\mu}_t|\nu) -\frac{1}{m}\int_M \sigma^m \,d\omega \nonumber\\
&\quad -\frac{1-\ve^{n+1}}{m(m-1)} \bigg\{
 \bigg( \frac{1-\ve^{n+1}}{\omega(B(y,\delta))} \bigg)^{m-1}
 -\frac{m}{\omega(B(y,\delta))} \int_{B(y,\delta)}\sigma^{m-1} \,d\omega
 \bigg\} \nonumber\\
&= \frac{\ve^{n+1}}{m(m-1)}\int_M
 \{ (\ve^{n+1}\rho_0)^{m-1}(\bJ^{\omega}_t)^{1-m}
 -m\sigma(\cT_t)^{m-1} \} \,d\mu_0. \label{eq:BA'}
\end{align}
Since $(\ve^{n+1}\rho_0)^{m-1}
 =\{ c_ne^{-\psi(\gamma(-\delta))}(1+a\delta)^n\ve^{-1}+O(1) \}^{1-m} \chi_{B_-}$,
the leading term of $(\ref{eq:BA'})$ (as $\ve \to 0$) is
\[ \frac{\ve^{m(n+1)}}{m(m-1)}\int_M
 \rho_0^{m-1}(\bJ^{\omega}_t)^{1-m} \,d\mu_0. \]
Therefore $(\ref{eq:mCD})$ with $t=1/2$ and the Jacobian equation (Theorem~\ref{th:MA}) yield that
\begin{align*}
\liminf_{\ve \to 0} \int_M (\bJ^{\omega}_{1/2})^{1-m} \,d\mu_0
&\ge \frac{1}{2} \big\{ \bJ^{\omega}_0\big( \gamma(-\delta) \big)^{1-m}
 +\bJ^{\omega}_1\big( \gamma(-\delta) \big)^{1-m} \big\} \\
&= \frac{1}{2}\bigg\{ 1+\bigg( \frac{1-a\delta}{1+a\delta} \bigg)^{n/N}
 e^{\{ \psi(\gamma(-\delta))-\psi(\gamma(\delta)) \}/N} \bigg\}.
\end{align*}
Combining this with \eqref{eq:DC+} and \eqref{eq:DC+2}, we obtain
\begin{align*}
&1+\frac{1}{2}\Ric(v) \delta^2 \\
&\ge (1+a\delta)^n e^{\psi(x)-\psi(\gamma(-\delta))}
 \bigg( \int_M (\bJ^{\omega}_t)^{1-m} \,d\mu_0 \bigg)^{1/(1-m)} +O(\delta^3) \\
&\ge \frac{1}{2^N}\Big\{ (1+a\delta)^{n/N} e^{\{\psi(x)-\psi(\gamma(-\delta))\}/N}
 +(1-a\delta)^{n/N} e^{\{\psi(x)-\psi(\gamma(\delta))\}/N} \Big\}^N +O(\delta^3).
\end{align*}
Hence we have, expanding the $(1/N)$-th power of both sides near $\delta=0$,
\begin{align*}
&1+\frac{1}{2N}\Ric(v)\delta^2 \\
&\ge \frac{1}{2}\Big\{ (1+a\delta)^{n/N} e^{\{\psi(x)-\psi(\gamma(-\delta))\}/N}
 +(1-a\delta)^{n/N} e^{\{\psi(x)-\psi(\gamma(\delta))\}/N} \Big\} +O(\delta^3) \\
&= 1+\frac{\delta^2}{2} \bigg[ \frac{n}{N}\bigg( \frac{n}{N}-1 \bigg) a^2
 -\bigg\{ \frac{(\psi \circ \gamma)''(0)}{N} -\frac{(\psi \circ \gamma)'(0)^2}{N^2}
 \bigg\} +\frac{2na}{N}\frac{(\psi \circ \gamma)'(0)}{N} \bigg]
+O(\delta^3) \\
&= 1+\frac{\delta^2}{2N}\bigg\{ -(\psi \circ \gamma)''(0)
 +\frac{n(n-N)}{N}a^2+\frac{2n(\psi \circ \gamma)'(0)}{N}a
 +\frac{(\psi \circ \gamma)'(0)^2}{N} \bigg\} +O(\delta^3).
\end{align*}
Therefore we obtain
\begin{equation}\label{eq:Rica}
\Ric(v)+(\psi \circ \gamma)''(0)
 -\frac{n(n-N)}{N}a^2-\frac{2n(\psi \circ \gamma)'(0)}{N}a
 -\frac{(\psi \circ \gamma)'(0)^2}{N} \ge 0.
\end{equation}
If $N>n$, then choosing the minimizer $a=(\psi \circ \gamma)'(0)/(N-n)$
gives the desired curvature bound
\[ \Ric_N(v) =\Ric(v) +(\psi \circ \gamma)''(0) -\frac{(\psi \circ \gamma)'(0)^2}{N-n} \ge 0. \]
If $N=n$, then we consider $a$ going to $\infty$ or $-\infty$
and find $(\psi \circ \gamma)'(0)=0$ as well as $\Ric_n(v) \ge 0$.

In the case of $m>1$, we use the same transport $(\ref{eq:mu0})$
and then the leading term of $(\ref{eq:BA})$ changes into
\[ \frac{1}{m(m-1)}\int_M \rho_0^{m-1}(\bJ^{\omega}_t)^{1-m} \,d\mu_0. \]
Thus calculations as above yield the reverse inequality of \eqref{eq:DC+} and finally
$(\ref{eq:Rica})$ with $N<0$.
We again choose the minimizer $a=(\psi \circ \gamma)'(0)/(N-n)$ and find $\Ric_N(v) \ge 0$.
Similarly, for the transport $(\ref{eq:mu1})$, the leading term
of $(\ref{eq:BA'})$ is
\[ \frac{\ve^{n+1}}{1-m}\int_M \sigma(\cT_t)^{m-1} \,d\mu_0, \]
and then $(\ref{eq:mCD})$ yields
$\Hess\Psi =\Hess(\sigma^{m-1}/(1-m)) \ge K$
(note that $W_2(\tilde{\mu}_0,\tilde{\mu}_1)^2=\ve^{n+1} W_2(\mu_0,\mu_1)^2$
has the same order).
$\qedd$
\end{proof}

\begin{remark}\label{rm:N<n}
(1) If we admit $m \in (0,(n-1)/n)$ and generalize $\Ric_N$ in $(\ref{eq:RicN})$
to $N \in (1,n)$, then Claim~\ref{cl:Jt} is false.
Moreover, as the coefficient of $a^2$ in $(\ref{eq:Rica})$ is negative,
$(\ref{eq:mCD})$ is never satisfied (let $a \to \infty$).
Compare this with \cite[$(1.7)$]{Stcon} which means $m \ge (n-1)/n$
in our setting.

(2) Note that the special case $\nu=\omega$ (i.e., $\Psi \equiv 0$)
in Theorem~\ref{th:mCD} makes sense only for $K=0$.
Then the assertion of Theorem~\ref{th:mCD} corresponds to the equivalence
between $\Ric_N \ge 0$ and the convexity of the R\'enyi entropy $S_N$,
i.e., the curvature-dimension condition $\CD(0,N)$ of $(M,\omega)$.

(3) In the limit case of $m=1$, two weights $\psi$ and $\Psi$ are synchronized
as $\nu=e^{-\psi-\Psi}\vol_g$, and $\Hess \Ent_{\nu} \ge K$
(i.e., $\CD(K,\infty)$ for $(M,\nu)$) is equivalent to the single condition
$\Ric +\Hess(\psi+\Psi) \ge K$ (\cite[Theorem~1]{vRS}, \cite[Proposition~4.14]{StI}).
For $m \neq 1$, however, $\psi$ and $\Psi$ keep separate and they measure
different phases of $(M,\omega,\nu)$, as indicated in Theorem~\ref{th:mCD}.
\end{remark}

\section{Functional inequalities}\label{sc:func}

Since Otto and Villani's celebrated work \cite{OV}, the displacement convexity
of entropy-type functionals has played a significant role in the study of
functional inequalities (and the concentration of measures).
In this section, we follow the argument in \cite[Section~6]{LV2}
that the direct application of the displacement convexity of the entropy
implies various functional inequalities.
Our proofs use only fundamental properties of convex functions.
In more analytic context, related results for $m \neq 1$ in the Euclidean spaces
$(M,\omega)=(\R^n,dx)$ can be found in \cite{AGK}, \cite{CGH} and \cite{Ta2}.
See especially \cite[Section~4]{AGK} and \cite[Section~3]{CGH} for various
generalizations of the Talagrand (transport) inequality,
logarithmic Sobolev (entropy-information) inequality, HWI inequality
and the Poincar\'e inequality.
The relation among these inequalities are also discussed there.

Throughout the section, we suppose that $m>1/2$, $\Ric_N \ge 0$
and that $\Hess \Psi \ge K$ holds for some $K>0$.
Note that $m>1/2$ is clear if $n \ge 3$.
Recall from Lemma~\ref{lm:K>0}(i), (iii) that $\nu(M)<\infty$ automatically
follows from these hypotheses, so that the normalization gives
\[ \bar{\nu}=\bar{\sigma}\omega=\exp_m(-\overline{\Psi})\omega
 :=\nu(M)^{-1} \nu \in \cP_{\ac}(M,\omega) \]
with $\Hess\overline{\Psi} \ge \nu(M)^{1-m}K$ according to Remark~\ref{rm:multi}.
Lemma~\ref{lm:K>0} moreover ensures that $\bar{\sigma} \in L^m(M,\omega)$,
$\bar{\nu} \in \cP^2_{\ac}(M,\omega)$ and that $\overline{M_0}$ is convex.
Keeping these in mind, we will consider $\nu$ with $\nu(M)=1$ for simplicity.

\begin{proposition}[Talagrand inequality]\label{pr:Ta}
Assume that $m \in [(n-1)/n,\infty) \setminus \{1/2,1\}$,
$\nu(M)=1$, $\Ric_N \ge 0$ and $\Hess \Psi \ge K>0$.
Then we have, for any $\mu \in \cP^2(\overline{M_0})$,
\[ W_2(\mu,\nu) \le \sqrt{\frac{2}{K}H_m(\mu|\nu)}. \]
\end{proposition}

\begin{proof}
There is nothing to prove if $H_m(\mu|\nu)=\infty$, so that
we assume $H_m(\mu|\nu)<\infty$.
Let $(\mu_t)_{t \in [0,1]} \subset \cP^2(\overline{M_0})$ be the optimal transport
from $\mu_0=\mu$ to $\mu_1=\nu$.
It follows from $(\ref{eq:mCD})$ and $H_m(\nu|\nu)=0$ that
\begin{equation}\label{eq:TaH}
H_m(\mu_t|\nu) \le (1-t)H_m(\mu|\nu)-\frac{K}{2}(1-t)tW_2(\mu,\nu)^2.
\end{equation}
Since $H_m(\mu_t|\nu) \ge 0$ (Lemma~\ref{lm:Hm}), we obtain
$H_m(\mu|\nu) \ge (K/2)W_2(\mu,\nu)^2$ by dividing $(\ref{eq:TaH})$
with $1-t$ and letting $t$ go to $1$.
$\qedd$
\end{proof}

The above Talagrand inequality is regarded as a comparison between
distances in Wasserstein geometry and information geometry
(recall Subsection~\ref{ssc:ent}).

In the remainder of the section, let $\Psi$ be locally Lipschitz.
For $\mu=\rho\omega \in \cP^2_{\ac}(M,\omega)$ such that $\rho$ is locally Lipschitz,
we define the {\it $m$-relative Fisher information} by
\begin{equation}\label{eq:Im}
I_m(\mu|\nu) :=\frac{1}{m^2}
 \int_M \big| \nabla[e'_m(\rho)-e'_m(\sigma)] \big|^2 \rho \,d\omega
 =\frac{1}{(m-1)^2}
 \int_M \big| \nabla(\rho^{m-1}-\sigma^{m-1}) \big|^2 \,d\mu.
\end{equation}
It will be demonstrated in Proposition~\ref{pr:dHm} that $\sqrt{I_m(\mu|\nu)}$
is the absolute gradient of $H_m(\cdot|\nu)$ at $\mu$.
Thus it is natural to expect that the convexity of $H_m(\cdot|\nu)$ yields
the following inequality.

\begin{theorem}[HWI and Logarithmic Sobolev inequalities]\label{th:LS}
We assume that $m \in [(n-1)/n,\infty) \setminus \{1/2,1\}$, $\nu(M)=1$, $\Ric_N \ge 0$,
$\Hess \Psi \ge K>0$ and that $\Psi$ is locally Lipschitz.
Then we have, for any $\mu=\rho\omega \in \cP^2_{\ac}(\overline{M_0},\omega)$
such that $H_m(\mu|\nu)<\infty$ and $\rho$ is Lipschitz,
\begin{align}
H_m(\mu|\nu)
&\le \sqrt{I_m(\mu|\nu)} \cdot W_2(\mu,\nu) -\frac{K}{2}W_2(\mu,\nu)^2,
 \label{eq:HWI} \\
H_m(\mu|\nu) &\le \frac{1}{2K}I_m(\mu|\nu). \label{eq:LS}
\end{align}
\end{theorem}

\begin{proof}
Let $\mu_t=\rho_t \omega \in \cP^2_{\ac}(\overline{M_0},\omega)$, $t \in [0,1]$,
be the optimal transport from $\mu_0=\mu$ to $\mu_1=\nu$ given by
$\mu_t=(\cT_t)_{\sharp}\mu$ with $\cT_t(x)=\exp_x(t\nabla\varphi(x))$,
and put $H(t):=H_m(\mu_t|\nu)$.
Then it follows from $(\ref{eq:TaH})$ that
\begin{equation}\label{eq:LSH}
H(0) \le \frac{H(0)-H(t)}{t}-\frac{K}{2}(1-t)W_2(\mu,\nu)^2.
\end{equation}
We shall estimate the term
\[ H(0)-H(t)=\frac{1}{m(m-1)}\int_M
 \{ (\rho^m-\rho_t^m)-m(\rho-\rho_t)\sigma^{m-1} \} \,d\omega. \]
Since the function $f(s):=s^m/(m-1)$ is convex, we have
\[ \frac{\rho^m-\rho_t^m}{m-1} \le f'(\rho)(\rho-\rho_t)
 =\frac{m}{m-1}\rho^{m-1}(\rho-\rho_t), \]
and hence
\[ H(0)-H(t) \le \frac{1}{m-1}\int_M
 (\rho^{m-1}-\sigma^{m-1})(\rho-\rho_t) \,d\omega. \]
As $(\cT_t)_{\sharp}\mu=\mu_t$, we observe
\[ \int_M (\rho^{m-1}-\sigma^{m-1}) \rho_t \,d\omega
 =\int_M \{ \rho(\cT_t)^{m-1}-\sigma(\cT_t)^{m-1} \} \,d\mu. \]
This yields
\[ H(0)-H(t) \le \frac{1}{m-1}\int_M \big\{
 (\rho^{m-1}-\sigma^{m-1})-\big( \rho(\cT_t)^{m-1}-\sigma(\cT_t)^{m-1} \big)
 \big\} \,d\mu. \]
Thus we obtain
\begin{align*}
&\limsup_{t \to 0}\frac{H(0)-H(t)}{t}
 \le \frac{1}{|m-1|}\int_M |\nabla(\rho^{m-1}-\sigma^{m-1})|
 \cdot d(\cT_0,\cT_1) \,d\mu \\
&\le \frac{1}{|m-1|}
 \bigg( \int_M |\nabla(\rho^{m-1}-\sigma^{m-1})|^2 \,d\mu \bigg)^{1/2}
 \bigg( \int_M d(\cT_0,\cT_1)^2 \,d\mu \bigg)^{1/2} \\
&= \sqrt{I_m(\mu|\nu)} \cdot W_2(\mu,\nu).
\end{align*}
Combining this with $(\ref{eq:LSH})$, we conclude that
\[ H_m(\mu|\nu)
 \le \sqrt{I_m(\mu|\nu)} \cdot W_2(\mu,\nu) -\frac{K}{2}W_2(\mu,\nu)^2
 \le \frac{1}{2K}I_m(\mu|\nu). \]
$\qedd$
\end{proof}

\begin{remark}\label{rm:func}
It is established in \cite{Ta2} that, in the Euclidean space
$(M,\omega)=(\R^n,dx)$, equality of $(\ref{eq:HWI})$ and $(\ref{eq:LS})$
is characterized by using $m$-Gaussian measures.
\end{remark}

We finally show a kind of Poincar\'e inequality.
Observe that letting $m=1$ recovers the usual global Poincar\'e inequality
$\int_M f^2 \,d\nu \le K^{-1}\int_M |\nabla f|^2 \,d\nu$.

\begin{proposition}[Global Poincar\'e inequality]\label{pr:Poin}
Assume that $(M,g)$ is compact, $m \in [(n-1)/n,\infty) \setminus \{1/2,1\}$,
$\nu(M)=1$, $\Ric_N \ge 0$, $\Hess \Psi \ge K>0$ and that $\Psi$ is Lipschitz.
Then, for any Lipschitz function $f:\overline{M_0} \lra \R$ such that
$\int_{\overline{M_0}} f \,d\nu=0$, we have
\[ \int_M f^2 \sigma^{m-1} \,d\nu
 \le \frac{1}{K}\int_M |\nabla( f\sigma^{m-1})|^2 \,d\nu. \]
\end{proposition}

\begin{proof}
Apply (\ref{eq:LS}) to $\mu=\rho \omega:=(1+\ve f)\sigma\omega$
for small $\ve >0$ and obtain
\[ \frac{1}{m(m-1)}
 \int_M \{ \rho^m-m\rho\sigma^{m-1}+(m-1)\sigma^m \} \,d\omega
 \le \frac{1}{2K} \frac{1}{(m-1)^2}
 \int_M |\nabla(\rho^{m-1}-\sigma^{m-1})|^2 \,d\mu. \]
We remark that $H_m(\mu|\nu)<\infty$ as $M$ is compact.
On the one hand,
\begin{align*}
\rho^m-m\rho\sigma^{m-1}+(m-1)\sigma^m
&= (1+\ve f)^m \sigma^m-m(1+\ve f) \sigma^m +(m-1)\sigma^m \\
&= \sigma^m\{ (1+\ve f)^m-1-m(\ve f) \} \\
&= m(m-1)\sigma^m \frac{f^2}{2}\ve^2 +O(\ve^3),
\end{align*}
where $O(\ve^3)$ is uniform on $M$ thanks to the compactness of $M$.
On the other hand,
\begin{align*}
|\nabla(\rho^{m-1}-\sigma^{m-1})|^2
&= \big| \nabla\big[ \big( (1+\ve f)^{m-1}-1 \big) \sigma^{m-1} \big] \big|^2 \\
&= |\nabla[(m-1)f\ve \sigma^{m-1}] +O(\ve^2) |^2 \\
&= (m-1)^2 \ve^2 |\nabla(f\sigma^{m-1})|^2 +O(\ve^3).
\end{align*}
Thus we have, letting $\ve$ go to zero,
\[ \int_M f^2 \sigma^m \,d\omega
 \le \frac{1}{K}\int_M |\nabla(f\sigma^{m-1})|^2 \,d\nu. \]
$\qedd$
\end{proof}

\section{Concentration of measures}\label{sc:conc}

This section is devoted to an application of Proposition~\ref{pr:Ta}
to the concentration of measures.
Let us assume $\nu(M)=1$ and define the {\it concentration function} by
\[ \alpha_{(M,\nu)}(r):=\sup \big\{ 1-\nu\big( B(A,r) \big) \,|\,
 A \subset M,\ \nu(A) \ge 1/2 \big\} \]
for $r>0$, where $A$ is any measurable set and
\[ B(A,r):=\{ y \in M \,|\, \inf_{x \in A}d(x,y)<r \}. \]
The function $\alpha_{(M,\nu)}$ describes how the probability measure
$\nu$ concentrates on the neighborhood of an arbitrary set of half
the total measure in a quantitative way
(in other words, a kind of large deviation principle).
An especially interesting situation is that a sequence $\{(M_i,\nu_i)\}_{i \in \N}$
satisfies $\lim_{i \to \infty}\alpha_{(M_i,\nu_i)}(r)=0$ for all $r>0$,
that means that $(M_i,\nu_i)$ is getting more and more concentrated.
We refer to \cite{Le} for the basic theory and applications
of the concentration of measure phenomenon.

In the classical case of $m=1$, it is well-known that the concentration
of measures has rich connections with functional inequalities
appearing in Section~\ref{sc:func}.
For instance, the $L^1$-transport inequality
$W_1(\mu,\nu) \le \sqrt{(2/K)\Ent_{\nu}(\mu)}$
implies the {\it normal concentration} $\alpha(r) \le Ce^{-cr^2}$ with
constants $c,C>0$ depending only on $K$ (\cite[Section~6.1]{Le}).
In the same spirit, we show that an application of Proposition~\ref{pr:Ta}
gives new examples of concentrating spaces.

We set $G_c=G_c(\nu):=\int_M \sigma^c \,d\omega$ for $c>1/2$.
Recall from Lemma~\ref{lm:K>0}(i) that, if $m<1$, $\Ric_N \ge 0$
and if $\Hess\Psi \ge K>0$, then
\begin{equation}\label{eq:Gc}
G_c(\nu) \le C_1(\omega)^{1-c}\nu(M)^c +C_2(m,c,\omega)K^{c/(m-1)}<\infty
\end{equation}
holds for each $c \in (1/2,1]$.

\begin{theorem}[$m<1$ case]\label{th:conc}
Let $(M,\omega)$ satisfy $\Ric_N \ge 0$ and $m \in [(n-1)/n,1) \cap (1/2,1)$.
\begin{enumerate}[{\rm (i)}]
\item Assume that $\nu(M)=1$ and $\Hess\Psi \ge K>0$.
Then we have
\begin{equation}\label{eq:conc}
\alpha_{(M,\nu)}(r)^{\theta -m} \ln_m\big( 2\alpha_{(M,\nu)}(r) \big)
 \le -G_{(m-\theta)/(1-\theta)}^{\theta-1}
 \bigg\{ \bigg( \sqrt{\frac{mK}{2}}r-\sqrt{G_m} \bigg)^2 -G_m \bigg\}
\end{equation}
for all $r>0$ and $\theta \in [0,2m-1)$.

\item Take a sequence $\nu_i=\exp_m(-\Psi_i)\omega \in \cP_{\ac}(M,\omega)$,
$i \in \N$, such that $\Hess\Psi_i \ge K_i$ and $\lim_{i \to \infty}K_i=\infty$.
Then we have $\lim_{i \to \infty}\alpha_{(M,\nu_i)}(r)=0$ for all $r>0$.
\end{enumerate}
\end{theorem}

\begin{proof}
(i) Note that $\nu \in \cP_{\ac}^2(M,\omega)$ by Lemma~\ref{lm:K>0}(ii) and $m>1/2$.
We also remark that $(\ref{eq:conc})$ clearly holds for
$r \le 2\sqrt{2G_m/mK}$.
Indeed, then the right-hand side is nonnegative and the trivial bound
$\alpha_{(M,\nu)}(r) \le 1/2$ implies $\ln_m(2\alpha_{(M,\nu)}(r)) \le 0$.

Suppose $r>2\sqrt{2G_m/mK}$, take a measurable set $A \subset M$ with
$\nu(A) \ge 1/2$ and put $B:=M \setminus B(A,r)$, $a:=\nu(A)$, $b:=\nu(B)$,
\[ \mu_A:=\frac{\chi_A}{a}\nu,\qquad \mu_B:=\frac{\chi_B}{b}\nu. \]
We assumed $b>0$ since there is nothing to prove if $b=0$ for all such $A$.
Observe that $W_1(\mu_A,\mu_B) \ge r$ as $d(x,y) \ge r$ for all $x \in A$ and $y \in B$.
The triangle inequality of $W_1$ and Proposition~\ref{pr:Ta} together imply
(as $W_1 \le W_2$ by the Schwarz inequality)
\[ r \le W_1(\mu_A,\mu_B) \le W_1(\mu_A,\nu)+W_1(\nu,\mu_B)
 \le \sqrt{\frac{2}{K}H_m(\mu_A|\nu)} +\sqrt{\frac{2}{K}H_m(\mu_B|\nu)}. \]
Note that
\[ H_m(\mu_A|\nu) =\frac{1}{m(1-m)}
 \int_A \frac{ma^{m-1}-1}{a^m} \sigma^m \,d\omega +\frac{1}{m}G_m \]
and $ma^{m-1}-1<0$ since $a \ge 1/2>m^{1/(1-m)}$.
Thus we obtain
\[ \sqrt{\frac{mK}{2}}r \le  \sqrt{G_m}
 +\sqrt{G_m+b^{-m} \frac{mb^{m-1}-1}{1-m} \int_B \sigma^m \,d\omega}. \]
We observe from $r>2\sqrt{2G_m/mK}$ that $\sqrt{mK/2}r>2\sqrt{G_m}$
which yields $0<mb^{m-1}-1 <(2b)^{m-1}-1$.
Hence we have
\begin{equation}\label{eq:con}
\bigg( \sqrt{\frac{mK}{2}}r-\sqrt{G_m} \bigg)^2 -G_m
 \le -b^{-m} \ln_m(2b) \int_B \sigma^m \,d\omega.
\end{equation}
It follows from the H\"older inequality that
\[ \int_B \sigma^m \,d\omega
 =\int_B \sigma^{\theta+(m-\theta)} \,d\omega
 \le \bigg( \int_B \sigma \,d\omega \bigg)^{\theta}
 \bigg( \int_B \sigma^{(m-\theta)/(1-\theta)} \,d\omega \bigg)^{1-\theta}
 \le b^{\theta} G_{(m-\theta)/(1-\theta)}^{1-\theta}, \]
where the assumption $\theta<2m-1$ ensures $(m-\theta)/(1-\theta) >1/2$.
Therefore we obtain the desired inequality $(\ref{eq:conc})$
by choosing $A_i \subset M$ such that
$\lim_{i \to \infty} \nu(M \setminus B(A_i,r)) =\alpha_{(M,\nu)}(r)$.

(ii) Thanks to $(\ref{eq:Gc})$, we know that
\[ \limsup_{i \to \infty} G_c(\nu_i) \le C_1(\omega)^{1-c} <\infty \]
for all $c \in (1/2,1]$.
Therefore we deduce from (i) with $\theta =0$ that, setting $\alpha_i:=\alpha_{(M,\nu_i)}(r)$,
\[ -\infty =\lim_{i \to \infty} \alpha_i^{-m} \ln_m(2\alpha_i)
 =-\lim_{i \to \infty} \frac{\alpha_i^{-1}}{2^{1-m}}
 \frac{1-(2\alpha_i)^{1-m}}{1-m} \]
which shows $\lim_{i \to \infty} \alpha_i=0$.
$\qedd$
\end{proof}

\begin{remark}\label{rm:conc}
(1) Taking the proof of Lemma~\ref{lm:K>0}(i) into account,
we can generalize Theorem~\ref{th:conc}(ii) as follows.
Suppose that a sequence $\{ (M_i,\omega_i,\nu_i) \}_{i \in \N}$ satisfies,
for $m \in [(n-1)/n,1) \cap (1/2,1)$,
\begin{enumerate}[(a)]
\item $\Ric_N \ge 0$ for all $(M_i,\omega_i)$,
\item $\nu_i=\exp_m(-\Psi_i)\omega_i \in \cP_{\ac}(M_i,\omega_i)$
so that $\Hess\Psi_i \ge K_i$ and $\lim_{i \to \infty}K_i=\infty$,
\item $\sup_{i \in \N} \omega_i(B(x_i,R))<\infty$ and
$\sup_{i \in \N} \area_{\omega_i}(S(x_i,R))<\infty$ for some $R>0$,
where $x_i \in M_i$ is the minimizer of $\Psi_i$.
\end{enumerate}
Then we have $\lim_{i \to \infty}\alpha_{(M_i,\nu_i)}(r)=0$ for all $r>0$.

(2) Taking the limit of $(\ref{eq:conc})$ as $m \to 1$ and then $\theta \to 1$,
we obtain
\[ \ln\big( 2\alpha(r) \big) \le -\bigg( \sqrt{\frac{K}{2}}r-1 \bigg)^2 +1. \]
Here $\lim_{c \to 1}G_c=G_1=1$ follows from the dominated convergence
theorem since $\sigma^c \le \max\{ \sigma,\sigma^{c_0} \} \in L^1(M,\omega)$
for $1/2<c_0 \le c<1$.
Therefore we recover the normal concentration
\[ \alpha(r) \le \frac{1}{2} \exp\bigg[ -\bigg( \sqrt{\frac{K}{2}}r-1 \bigg)^2+1 \bigg]
 \le \frac{1}{2} e^{-Kr^2/4+2} \]
which is well-known to hold for $(M,\omega)$ with $\Ric_{\infty} \ge K>0$.
\end{remark}

Theorem~\ref{th:conc}(ii) is applicable to the fundamental example of
$m$-Gaussian measures (see Example~\ref{ex:Nm}).

\begin{example}\label{ex:conc}
Let $\{ N_m(v_i,V_i) \}_{i \in \N} \subset \cP_{\ac}^2(\R^n,dx)$ be a sequence of
$m$-Gaussian measures with $m \in [(n-1)/n,1) \cap (1/2,1)$ satisfying
\[ \lim_{i \to \infty} (\det V_i)^{(1-m)/2}\Lambda_i^{-1}=\infty,  \]
where $\Lambda_i$ is the largest eigenvalue of $V_i$.
Then we have $\lim_{i \to \infty}\alpha_{(\R^n,N_m(v_i,V_i))}(r)=0$
for all $r>0$.
Note that $(\det V_i)^{(1-m)/2} \Lambda_i^{-1} \le \Lambda_i^{(1-m)n/2-1} \le \Lambda_i^{-1/2}$.
\end{example}

Under the additional assumption that $\omega(M)<\infty$,
we further obtain the {\it $m$-normal concentration}.
We first prove a computational lemma for later use.

\begin{lemma}\label{lm:conc}
\begin{enumerate}[{\rm (i)}]
\item For any $m \in (1/2,1)$ and $a,r>0$, we have
\[ \exp_m\big( -(ar-1)^2+1 \big)
 \le (2m-1)^{1/(m-1)} \exp_m\bigg( -\frac{a^2}{2}r^2 \bigg). \]

\item For any $m \in (1,2)$ and $a,r>0$, we have
\[ \exp_m\big( (ar-1)^2-1 \big)
 \ge \bigg( \frac{2}{m}-1 \bigg)^{1/(m-1)} \exp_m\bigg( \frac{a^2}{2}r^2 \bigg). \]
\end{enumerate}
\end{lemma}

\begin{proof}
(i) We just calculate
\begin{align*}
&\exp_m\big( -(ar-1)^2+1 \big) \le \exp_m\bigg( -\frac{a^2}{2}r^2+2 \bigg) \\
&= \bigg\{ 1+(m-1)\bigg( -\frac{a^2}{2}r^2+2 \bigg) \bigg\}^{1/(m-1)} \\
&= (2m-1)^{1/(m-1)} \bigg\{ 1+(m-1)
 \bigg( -\frac{a^2}{2(2m-1)}r^2 \bigg) \bigg\}^{1/(m-1)} \\
&\le (2m-1)^{1/(m-1)} \exp_m\bigg( -\frac{a^2}{2}r^2 \bigg).
\end{align*}

(ii) We similarly find
\begin{align*}
&\exp_m\big( (ar-1)^2-1 \big)
 \ge \exp_m\bigg[ \bigg(1-\frac{m}{2} \bigg) a^2r^2 -\frac{2}{m} \bigg] \\
&= \bigg\{ \bigg( \frac{2-m}{m} \bigg)
 +(m-1)\bigg( 1-\frac{m}{2} \bigg) a^2r^2 \bigg\}^{1/(m-1)} \\
&= \bigg( \frac{2}{m}-1 \bigg)^{1/(m-1)}
 \bigg\{ 1+\frac{m}{2}(m-1)a^2r^2 \bigg\}^{1/(m-1)} \\
&\ge \bigg( \frac{2}{m}-1 \bigg)^{1/(m-1)} \exp_m\bigg( \frac{a^2}{2}r^2 \bigg).
\end{align*}
Note that the hypothesis $m \in (1,2)$ ensures that
\[ \bigg(1-\frac{m}{2} \bigg) a^2r^2 -\frac{2}{m} >-\frac{2}{m}
 >-\frac{1}{m-1}. \]
$\qedd$
\end{proof}

\begin{corollary}[$m$-normal concentration]\label{cr:conc}
Assume that $m \in [(n-1)/n,1) \cap (1/2,1)$, $\nu(M)=1$, $\omega(M)<\infty$,
$\Ric_N \ge 0$ and $\Hess\Psi \ge K>0$.
Then we have
\[ \alpha_{(M,\nu)}(r) \le \frac{(2m-1)^{1/(m-1)}}{2}
 \exp_m\bigg( -\frac{mK}{4\omega(M)^{1-m}}r^2 \bigg) \]
for all $r>0$.
\end{corollary}

\begin{proof}
Let us use the same notation as the proof of Theorem~\ref{th:conc}.
We deduce from the H\"older inequality that
\[ \int_B \sigma^m \,d\omega \le \bigg( \int_B \sigma \,d\omega \bigg)^m
 \omega(B)^{1-m} =b^m \omega(B)^{1-m} \le b^m \omega(M)^{1-m} \]
and, similarly, $G_m \le \omega(M)^{1-m}$.
In particular, $r^2>8\omega(M)^{1-m}/mK$
(otherwise the assertion is clear since $(2m-1)^{1/(m-1)} \exp_m(-2)>1$)
implies $r^2>8G_m/mK$.
Therefore we deduce from $(\ref{eq:con})$ that
\begin{align*}
&\bigg( \sqrt{\frac{mK}{2}}r-\omega(M)^{(1-m)/2} \bigg)^2 -\omega(M)^{1-m}
\le \bigg( \sqrt{\frac{mK}{2}}r-\sqrt{G_m} \bigg)^2 -G_m \\
&\le -b^{-m} \ln_m(2b) \int_B \sigma^m \,d\omega
 \le -\omega(M)^{1-m} \ln_m(2b),
\end{align*}
and hence
\[ \alpha_{(M,\nu)}(r) \le \frac{1}{2} \exp_m\bigg[ -\bigg( 
 \omega(M)^{(m-1)/2} \sqrt{\frac{mK}{2}}r-1 \bigg)^2 +1 \bigg]. \]
Then Lemma~\ref{lm:conc}(i) completes the proof.
$\qedd$
\end{proof}

\begin{remark}\label{rm:expm}
Note that, for $m<1$, $\exp_m(-cr^2)$ is greater than $e^{-cr^2}$ and is
a polynomial of $r$, so that the $m$-normal concentration is weaker than
the exponential concentration $\alpha(r) \le Ce^{-cr}$.
This is natural and the most we can expect, because the $m$-Gaussian
measures have only the polynomial decay.
\end{remark}

For $m>1$, Lemma~\ref{lm:K>0}(iii) ensures that $\supp\nu$ is bounded.
Thus $\|\sigma\|_{\infty}<\infty$ and $G_c(\nu)<\infty$ for all $c>0$.
Then the proof of Theorem~\ref{th:conc}(i) is applicable to $m \in (1,2]$ and
gives the same estimate $(\ref{eq:conc})$ for all $r>0$ and $\theta \in [0,1)$.
Furthermore, for $m<2$, we again obtain the $m$-normal concentration
(depending on $\|\sigma\|_{\infty}$).

\begin{proposition}[$m>1$ case]\label{pr:conc}
Let $(M,\omega)$ satisfy $\Ric_N \ge 0$ and $m \in (1,2]$.
\begin{enumerate}[{\rm (i)}]
\item Assume that $\nu(M)=1$ and $\Hess\Psi \ge K>0$.
Then we have $(\ref{eq:conc})$ for all $r>0$ and $\theta \in [0,1)$.

\item If in addition $m<2$, then we have
\[ \alpha_{(M,\nu)}(r)^{-1} \ge \bigg( \frac{2}{m}-1 \bigg)^{1/(m-1)}
 \exp_m\bigg( \frac{mK\|\sigma\|_{\infty}^{1-m}}{4}r^2 \bigg) \]
for all $r>0$.
\end{enumerate}
\end{proposition}

\begin{proof}
(i) This is completely the same as Theorem~\ref{th:conc}(i),
since $1/2 \ge m^{1/(1-m)}$ holds also for $m \in (1,2]$.

(ii) In $(\ref{eq:con})$ (with $m>1$), we observe
$\int_B \sigma^m \,d\omega \le b\|\sigma\|_{\infty}^{m-1}$ and
$G_m \le \|\sigma\|_{\infty}^{m-1}$.
Note also that $r^2>8\|\sigma\|_{\infty}^{m-1}/mK$
(otherwise $((2-m)/m)^{1/(m-1)}\exp_m(2)<1$ immediately gives the assertion)
ensures $r^2>8G_m/mK$.
These yield
\[ \bigg( \sqrt{\frac{mK}{2}}r -\|\sigma\|_{\infty}^{(m-1)/2} \bigg)^2
 -\|\sigma\|_{\infty}^{m-1}
 \le -b^{1-m}\|\sigma\|_{\infty}^{m-1} \ln_m(2b)
 \le \|\sigma\|_{\infty}^{m-1} \ln_m(b^{-1}). \]
Hence we have
\[ \alpha_{(M,\nu)}(r)^{-1} \ge \exp_m\bigg[ \bigg(
 \|\sigma\|_{\infty}^{(1-m)/2} \sqrt{\frac{mK}{2}}r-1 \bigg)^2 -1 \bigg], \]
and Lemma~\ref{lm:conc}(ii) completes the proof.
$\qedd$
\end{proof}

Note that we obtained the estimate of the form $\alpha(r) \le C\exp_m(-cr^2)$
for $m<1$, while $\alpha(r) \le C\{\exp_m(cr^2)\}^{-1}$ for $m>1$.
This is in a sense natural because the domain of $\exp_m$ is $(-\infty,1/(1-m))$
for $m<1$ and $[-1/(m-1),\infty)$ for $m>1$.

\begin{remark}\label{rm:cm>1}
We deduce from Proposition~\ref{pr:conc}(ii) that,
if $\lim_{i \to \infty} K_i\|\sigma_i\|_{\infty}^{1-m}=\infty$ for some sequence
$\{(M_i,\nu_i)\}_{i \in \N}$ satisfying $\Hess\Psi_i \ge K_i$,
then we have $\lim_{i \to \infty}\alpha_{(M_i,\nu_i)}(r)=0$ for all $r>0$
(e.g., a sequence of $m$-Gaussian measures $\{N_m(v_i,V_i)\}_{i \in \N}$
such that $\lim_{i \to \infty} \Lambda_i =0$, compare this with Example~\ref{ex:conc}).
This is, however, an immediate consequence of a stronger conclusion
$\lim_{i \to \infty}\diam(\supp\nu_i)=0$ of Lemma~\ref{lm:K>0}(iii) (valid for all $m>1$).
Indeed,
\[ \diam(\supp\nu_i)^2 \le \frac{8}{K_i}\bigg( \frac{1}{m-1}-\inf_{M_i}\Psi_i \bigg)
 =\frac{8}{K_i} \frac{\|\sigma_i\|_{\infty}^{m-1}}{m-1}. \]
\end{remark}

\section{Gradient flow of $H_m$}\label{sc:gf}

In this section, we show that the gradient flow of the $m$-relative entropy
produces a weak solution to the porous medium equation ($m>1$)
or the fast diffusion equation ($m<1$).
This kind of interpretation of evolution equations has turned out extremely
useful after the pioneering work due to Jordan et al.~\cite{JKO}.
There are several ways of explaining this coincidence
(see, e.g., \cite{JKO}, \cite{AGS} and \cite[Chapter~23]{Vi2}),
among them, here we follow the rather `metric geometric' approach in \cite{Ogra}.
To do this, we start with a review of the geometric structure of
the Wasserstein space and the general theory of gradient flows in it
in accordance with the strategy in \cite{Ogra} (see also \cite{GO}).
Throughout the section, $(M,g)$ is assumed to be compact,
so that $\cP^2(M)=\cP(M)$ and $\sigma \in L^m(M,\omega)$.

\subsection{Geometric structure of  $(\cP(M),W_2)$}\label{ssc:gf1}

We briefly review the geometric structure of $(\cP(M),W_2)$.
It is known that $(\cP(M),W_2)$ is an Alexandrov space of nonnegative
curvature if and only if $(M,g)$ has the nonnegative sectional curvature
(\cite[Proposition~2.10]{StI}, \cite[Theorem~A.8]{LV2}).
In the case where $(M,g)$ is not nonnegatively curved, although $(\cP(M),W_2)$
does not admit any lower curvature bound (\cite[Proposition~2.10]{StI}),
we can show the following (see also \cite[Theorem~3.6]{Ogra}).

\begin{theorem}\label{th:angle}{\rm (\cite[Theorem~3.4, Remark~3.5]{Gi})}
Given $\mu \in \cP(M)$ and unit speed geodesics
$\alpha,\beta:[0,\delta) \lra \cP(M)$ with $\alpha(0)=\beta(0)=\mu$, the joint limit
\[ \lim_{s,t \to 0}\frac{s^2 +t^2-W_2(\alpha(s),\beta(t))^2}{2st}\ \in [-1,1] \]
exists.
\end{theorem}

Theorem~\ref{th:angle} means that an angle between $\alpha$ and $\beta$ makes sense,
so that $(\cP(M),W_2)$ looks like a Riemannian space
(rather than a Finsler space), and we can investigate its infinitesimal structure
in the manner of the theory of Alexandrov spaces.
For $\mu \in \cP(M)$, denote by $\Sigma'_{\mu}[\cP(M)]$ the set of all
(nontrivial) unit speed minimal geodesics emanating from $\mu$.
Given $\alpha,\beta \in \Sigma'_{\mu}[\cP(M)]$, Theorem~\ref{th:angle}
verifies that the {\it angle}
\[ \angle_{\mu}(\alpha,\beta)
 :=\arccos \bigg( \lim_{s,t \to 0}\frac{s^2 +t^2-W_2(\alpha(s),\beta(t))^2}{2st} \bigg)
 \in [0,\pi] \]
is well-defined.
We define the {\it space of directions} $(\Sigma_{\mu}[\cP(M)],\angle_{\mu})$
as the completion of $(\Sigma'_{\mu}[\cP(M)]/\!\!\sim, \angle_{\mu})$,
where $\alpha \sim \beta$ holds if $\angle_{\mu}(\alpha,\beta)=0$.
The {\it tangent cone} $(C_{\mu}[\cP(M)],\sigma_{\mu})$ is defined as
the Euclidean cone over $(\Sigma_{\mu}[\cP(M)],\angle_{\mu})$, i.e.,
\begin{align*}
C_{\mu}[\cP(M)] &:=\big( \Sigma_{\mu}[\cP(M)] \times [0,\infty) \big)
 \big/ \big( \Sigma_{\mu}[\cP(M)] \times \{0\} \big), \\
\sigma_{\mu}\big( (\alpha,s),(\beta,t) \big)
&:=\sqrt{s^2+t^2-2st\cos\angle_{\mu}(\alpha,\beta)}.
\end{align*}
Using this infinitesimal structure, we introduce a class of `differentiable curves'.

\begin{definition}[Right differentiability]\label{df:rd}
We say that a curve $\xi:[0,l) \lra \cP(M)$ is {\it right differentiable} at $t \in [0,l)$
if there is $\bv \in C_{\xi(t)}[\cP(M)]$ such that, for any sequences
$\{\ve_i\}_{i \in \N}$ of positive numbers tending to zero and $\{ \alpha_i \}_{i \in \N}$
of unit speed minimal geodesics from $\xi(t)$ to $\xi(t+\ve_i)$, the sequence
$\{ (\alpha_i,W_2(\xi(t),\xi(t+\ve_i))/\ve_i) \}_{i \in \N} \subset C_{\xi(t)}[\cP(M)]$
converges to $\bv$.
Such $\bv$ is clearly unique if it exists, and then we write $\dot{\xi}(t)=\bv$.
\end{definition}

\subsection{Gradient flows in $(\cP(M),W_2)$}\label{ssc:gf2}

Consider a lower semi-continuous function $f:\cP(M) \lra (-\infty,+\infty]$
which is $K$-convex in the weak sense for some $K \in \R$.
We in addition suppose that $f$ is not identically $+\infty$, and define
$\cP^*(M):=\{ \mu \in \cP(M) \,|\, f(\mu)<\infty \}$.

Given $\mu \in \cP^*(M)$ and $\alpha \in \Sigma_{\mu}[\cP(M)]$, we set
\[ D_{\mu}f(\alpha):=\liminf_{\Sigma'_{\mu}[\cP(M)] \ni \beta \to \alpha}
 \lim_{t \to 0}\frac{f(\beta(t))-f(\mu)}{t}. \]
Define the {\it absolute gradient} (called the {\it local slope} in \cite{AGS})
of $f$ at $\mu \in \cP^*(M)$ by
\[ |\grad f|(\mu):=\max\bigg\{ 0, \limsup_{\tilde{\mu} \to \mu}
 \frac{f(\mu)-f(\tilde{\mu})}{W_2(\mu,\tilde{\mu})} \bigg\}. \]
Note that $-D_{\mu}f(\alpha) \le |\grad f|(\mu)$ for any $\alpha \in \Sigma_{\mu}[\cP(M)]$.

\begin{lemma}\label{lm:gv}{\rm (\cite[Lemma~4.2]{Ogra})}
For each $\mu \in \cP^*(M)$ with $0<|\grad f|(\mu)<\infty$,
there exists unique $\alpha \in \Sigma_{\mu}[\cP^*(M)]$ satisfying
$D_{\mu}f(\alpha)=-|\grad f|(\mu)$.
\end{lemma}

Using $\alpha$ in the above lemma, we define the {\it negative gradient vector}
of $f$ at $\mu$ as
\[ \grad f(\mu):=\big( \alpha,|\grad f|(\mu) \big) \in C_{\mu}[\cP(M)]. \]
In the case of $|\grad f|(\mu)=0$, we simply define $\grad f(\mu)$ as the origin
of $C_{\mu}[\cP(M)]$.

\begin{definition}[Gradient curves]\label{df:gc}
A continuous curve $\xi:[0,l) \lra \cP^*(M)$ which is locally Lipschitz on $(0,l)$
is called a {\it gradient curve} of $f$ if $|\grad f|(\xi(t))<\infty$
for all $t \in (0,\infty)$ and if it is right differentiable with
$\dot{\xi}(t)=\grad f(\xi(t))$ at all $t \in (0,l)$.
We say that a gradient curve $\xi$ is {\it complete} if it is defined on entire $[0,\infty)$.
\end{definition}

\begin{theorem}\label{th:cont}
{\rm (\cite[Theorem~5.11, Corollary~6.3]{Ogra}, \cite[Theorem~4.2]{GO})}
\begin{enumerate}[{\rm (i)}]
\item From any $\mu \in \cP^*(M)$, there starts a unique complete gradient curve
$\xi:[0,\infty) \lra \cP^*(M)$ of $f$ with $\xi(0)=\mu$.

\item Given any two gradient curves $\xi,\zeta:[0,\infty) \lra \cP^*(M)$ of $f$,
we have
\begin{equation}\label{eq:cont}
W_2 \big(\xi(t),\zeta(t) \big) \le e^{-Kt}W_2\big( \xi(0),\zeta(0) \big)
\end{equation}
for all $t \in [0,\infty)$.
\end{enumerate}
\end{theorem}

To be precise, the uniqueness in (i) is a consequence of
the $K$-contraction property \eqref{eq:cont}.
Therefore the {\it gradient flow} $G:[0,\infty) \times \cP^*(M) \lra \cP^*(M)$
of $f$, given as $G(t,\mu)=\xi(t)$ in Theorem~\ref{th:cont}(i),
is uniquely determined and extended to the closure
$G:[0,\infty) \times \overline{\cP^*(M)} \lra \overline{\cP^*(M)}$ continuously.

\subsection{$m$-relative entropy and the porous medium/fast diffusion equation}

We recall basic notions of calculus on weighted Riemannian manifolds
$(M,\omega)$ with $\omega=e^{-\psi}\vol_g$.
For a $C^1$-vector field $V$ on $M$, we define
the {\it weighted divergence} as
\[ \div_{\omega}V:=\div V-\langle V,\nabla\psi \rangle, \]
where $\div V$ denotes the usual divergence of $V$ for $(M,\vol_g)$.
Note that, for any $f \in C^1(M)$,
\begin{align*}
\int_M \langle \nabla f,V \rangle \,d\omega
&=\int_M \langle \nabla f,e^{-\psi}V \rangle \,d\!\vol_g
 =-\int_M f\div(e^{-\psi}V) \,d\!\vol_g \\
&= -\int_M f\div_{\omega}V \,d\omega.
\end{align*}
For $f \in C^2(M)$, the {\it weighted Laplacian} is defined by
\[ \Delta^{\omega}f:=\div_{\omega}(\nabla f)
 =\Delta f -\langle \nabla f, \nabla\psi \rangle. \]
Then it is an established fact that the gradient flow of the corresponding relative entropy
(or the {\it free energy})
\[ \Ent_{\omega}(\rho\omega) =\int_M \rho \ln\rho \,d\omega
 =\int_M (\rho e^{-\psi}) \ln(\rho e^{-\psi}) \,d\!\vol_g +\int_M \psi \,d\mu \]
produces a solution to the associated {\it heat equation}
(or the {\it Fokker-Planck equation})
\[ \frac{\del\rho}{\del t} =\Delta^{\omega}\rho
 =e^{\psi}\big\{ \Delta(\rho e^{-\psi})
 +\div\big( (\rho e^{-\psi})\nabla\psi \big) \big\}. \]
See \cite[Theorem~5.1]{JKO}, \cite[Subsection~8.4.2]{Vi1} for the Euclidean case,
\cite[Theorem~6.6]{Ogra}, \cite[Theorem~4.6]{GO}, \cite[Corollary~23.23]{Vi2}
for the Riemannian case, and \cite[Section~7]{OS} for the Finsler case.

We shall see that a similar argumentation gives a weak solution to
the {\it porous medium equation} for $m>1$ or the {\it fast diffusion equation}
for $m<1$ (with drift) of the form
\begin{equation}\label{eq:pme}
\frac{\del\rho}{\del t} =\frac{1}{m}\Delta^{\omega}(\rho^m)
 +\div_{\omega}(\rho \nabla\Psi)
\end{equation}
as gradient flow of the $m$-relative entropy $H_m(\cdot|\nu)$.
This is demonstrated by Otto~\cite{Ot} for the Tsallis entropy as well as
$H_m(\cdot|N_m(0,cI_n))$ with respect to the $m$-Gaussian measures
$N_m(0,cI_n)$ on $(\R^n,dx)$, and by Villani~\cite[Theorem~23.19]{Vi2}
on weighted Riemannian manifolds in a different way of interpretation from ours.
Here we present a precise proof along the strategy of \cite{Ogra}, \cite{GO}.
Recall that $\nu=\exp_m(-\Psi)\omega$.

\begin{theorem}[Gradient flow of $H_m$]\label{th:gf}
Let $(M,g)$ be compact, $m \in ((n-1)/n,1) \cup (1,2]$ and $\Psi$ be Lipschitz.
If a curve $(\mu_t)_{t \in [0,\infty)} \subset \cP_{\ac}(M,\omega)$ is a gradient
curve of $H_m(\cdot|\nu)$, then its density function $\rho_t$ is a weak solution
to the porous medium or the fast diffusion equation $(\ref{eq:pme})$.
To be precise,
\begin{equation}\label{eq:wpme}
\int_M \phi_{t_1} \,d\mu_{t_1} -\int_M \phi_{t_0} \,d\mu_{t_0}
 = \int_{t_0}^{t_1} \int_M \bigg\{ \frac{\del \phi_t}{\del t}
 +\frac{1}{m}\rho_t^{m-1}\Delta^{\omega}\phi_t
 +\frac{1}{m-1}\langle \nabla \phi_t,\nabla(\sigma^{m-1}) \rangle \bigg\} \,d\mu_t \,dt
\end{equation}
holds for all $0\le t_0<t_1<\infty$ and $\phi \in C^{\infty}(\R \times M)$,
where $\mu_t=\rho_t \omega$, $\phi_t=\phi(t,\cdot)$.
\end{theorem}

\begin{proof}
Fix $t \in (0,\infty)$ and, given small $\delta>0$, choose
$\mu^{\delta} \in \cP(M)$ as a minimizer of the function
\[ \mu\ \longmapsto\ H_m(\mu|\nu) +\frac{W_2(\mu,\mu_t)^2}{2\delta}. \]
We postpone the proof of the following technical claim until the end of the section.
We remark that the hypotheses $m>(n-1)/n$ and $m \le 2$ come into play
in the proof of Claim~\ref{cl:gf}(i) and (iii), respectively.

\begin{claim}\label{cl:gf}
\begin{enumerate}[{\rm (i)}]
\item Such $\mu^{\delta}$ indeed exists and is absolutely continuous
with respect to $\omega$.

\item We have
\[ \lim_{\delta \to 0} \frac{W_2(\mu^{\delta},\mu_t)^2}{2\delta}=0, \qquad
 \lim_{\delta \to 0} H_m(\mu^{\delta}|\nu)=H_m(\mu_t|\nu). \]
In particular, $\mu^{\delta}$ converges to $\mu_t$ weakly.

\item Moreover, by putting $\mu^{\delta}=\rho^{\delta}\omega$,
$(\rho^{\delta})^m$ converges to $\rho_t^m$ in $L^1(M,\omega)$.
\end{enumerate}
\end{claim}

Take a Lipschitz function $\varphi:M \lra \R$ such that
$\cT(x):=\exp_x(\nabla\varphi(x))$ gives the optimal transport
from $\mu^{\delta}$ to $\mu_t$.
We consider the transport $\mu^{\delta}_{\ve}:=(\cF_{\ve})_{\sharp}\mu^{\delta}$
in another direction for small $\ve>0$, where $\cF_{\ve}(x):=\exp_x(\ve\nabla \phi_t(x))$.
It immediately follows from the choice of $\mu^{\delta}$ that
\begin{equation}\label{eq:gf1}
H_m(\mu^{\delta}_{\ve}|\nu)
 +\frac{W_2(\mu^{\delta}_{\ve},\mu_t)^2}{2\delta}
 \ge H_m(\mu^{\delta}|\nu) +\frac{W_2(\mu^{\delta},\mu_t)^2}{2\delta}.
\end{equation}
We first estimate the difference of distances.
Observe that, as $(\cF_{\ve} \times \cT)_{\sharp}\mu^{\delta}$ is a
(not necessarily optimal) coupling of $\mu^{\delta}_{\ve}$ and $\mu_t$,
\begin{align*}
&\limsup_{\ve \to 0}
 \frac{W_2(\mu^{\delta}_{\ve},\mu_t)^2-W_2(\mu^{\delta},\mu_t)^2}{\ve} \\
&\le \limsup_{\ve \to 0}\frac{1}{\ve}\int_M \big\{
 d\big( \cF_{\ve}(x),\cT(x) \big)^2-d\big( x,\cT(x) \big)^2 \big\} \,d\mu^{\delta}(x) \\
&= -\int_M 2\langle \nabla \phi_t,\nabla\varphi \rangle \,d\mu^{\delta}.
\end{align*}
We used the first variation formula for the distance $d$ in the last line
(cf.\ \cite[Theorem~II.4.1]{Ch}).
Thanks to the compactness of $M$, there is a constant $C>0$ such that
\[ \phi_t\big( \cT(x) \big) \le \phi_t(x)+\langle \nabla \phi_t(x),\nabla\varphi(x) \rangle
 +Cd\big( x,\cT(x) \big)^2. \]
Thus we obtain, by virtue of Claim~\ref{cl:gf}(ii),
\begin{align*}
&\liminf_{\delta \to 0} \frac{1}{2\delta} \limsup_{\ve \to 0}
 \frac{W_2(\mu^{\delta}_{\ve},\mu_t)^2-W_2(\mu^{\delta},\mu_t)^2}{\ve}
 \le -\limsup_{\delta \to 0} \frac{1}{\delta} \int_M
 \langle \nabla \phi_t,\nabla\varphi \rangle \,d\mu^{\delta} \\
&\le \liminf_{\delta \to 0} \frac{1}{\delta} \bigg[ \int_M
 \{ \phi_t-\phi_t(\cT) \} \,d\mu^{\delta} +CW_2(\mu^{\delta},\mu_t)^2 \bigg] \\
&= \liminf_{\delta \to 0} \frac{1}{\delta} \bigg\{
 \int_M \phi_t \,d\mu^{\delta} -\int_M \phi_t \,d\mu_t \bigg\}.
\end{align*}

Next we calculate the difference of entropies in $(\ref{eq:gf1})$.
We put $\mu^{\delta}=\rho^{\delta}\omega$,
$\mu^{\delta}_{\ve}=\rho^{\delta}_{\ve} \omega$ and
$\bJ^{\omega}_{\ve}:=e^{\psi-\psi(\cF_{\ve})}\det(D\cF_{\ve})$.
Then we obtain from the Jacobian equation
$\rho^{\delta}_{\ve}(\cF_{\ve})\bJ^{\omega}_{\ve}=\rho^{\delta}$
(Theorem~\ref{th:MA}) that
\begin{align*}
&H_m(\mu^{\delta}_{\ve}|\nu) -\frac{1}{m}\int_M \sigma^m \,d\omega
 = \frac{1}{m(m-1)} \int_M \{ (\rho^{\delta}_{\ve})^{m-1}
 -m\sigma^{m-1} \} \,d\mu^{\delta}_{\ve} \\
&= \frac{1}{m(m-1)} \int_M \{ \rho^{\delta}_{\ve}(\cF_\ve)^{m-1}
 -m\sigma(\cF_{\ve})^{m-1} \} \,d\mu^{\delta} \\
&= \frac{1}{m(m-1)} \int_M \bigg\{ \bigg(
 \frac{\rho^{\delta}}{\bJ^{\omega}_{\ve}} \bigg)^{m-1}
 -m\sigma(\cF_{\ve})^{m-1} \bigg\} \,d\mu^{\delta}.
\end{align*}
Thus we have
\begin{align*}
&H_m(\mu^{\delta}|\nu) -H_m(\mu^{\delta}_{\ve}|\nu) \\
&=\frac{1}{m(m-1)} \int_M \Big[ (\rho^{\delta})^{m-1} \{
 1-(\bJ^{\omega}_{\ve})^{1-m} \} -m\{ \sigma^{m-1}-\sigma(\cF_{\ve})^{m-1}
 \} \Big] \,d\mu^{\delta}.
\end{align*}
Note that, as $\det(D\cF_0)=1$,
\begin{align*}
\lim_{\ve \to 0}\frac{\bJ^{\omega}_{\ve}-1}{\ve}
&=\lim_{\ve \to 0}\frac{e^{\psi-\psi(\cF_{\ve})}\det(D\cF_{\ve})-1}{\ve}
 =\trace(\Hess\phi_t) -\langle \nabla \phi_t,\nabla\psi \rangle \\
&=\Delta \phi_t -\langle \nabla \phi_t,\nabla\psi \rangle
 =\Delta^{\omega} \phi_t.
\end{align*}
Hence we obtain, together with Claim~\ref{cl:gf}(iii),
\begin{align}
&\lim_{\ve \to 0}
 \frac{H_m(\mu^{\delta}|\nu)-H_m(\mu^{\delta}_{\ve}|\nu)}{\ve} \nonumber\\
&= \int_M \bigg\{ \frac{1}{m}(\rho^{\delta})^{m-1}\Delta^{\omega}\phi_t
 +\frac{1}{m-1}\langle \nabla \phi_t,\nabla(\sigma^{m-1}) \rangle \bigg\}
 \,d\mu^{\delta} \label{eq:gf2}\\
&\to \int_M \bigg\{ \frac{1}{m}\rho_t^{m-1}\Delta^{\omega}\phi_t
 +\frac{1}{m-1}\langle \nabla \phi_t,\nabla(\sigma^{m-1}) \rangle \bigg\}
 \,d\mu_t \nonumber
\end{align}
as $\delta$ tends to zero.

These together imply
\[ \liminf_{\delta \to 0} \frac{1}{\delta} \bigg\{
 \int_M \phi_t \,d\mu^{\delta} -\int_M \phi_t \,d\mu_t \bigg\}
 \ge \int_M \bigg\{ \frac{1}{m}\rho_t^{m-1}\Delta^{\omega}\phi_t
 +\frac{1}{m-1}\langle \nabla \phi_t,\nabla(\sigma^{m-1}) \rangle \bigg\} \,d\mu_t. \]
Moreover, equality holds since we can change $\phi$ into $-\phi$.
Recall from \cite[(4)]{GO} (see also \cite[Lemma~6.4]{Ogra}) that
\[ \lim_{\delta \to 0} \frac{1}{\delta}\bigg\{ \int_M h \,d\mu_{t+\delta}
 -\int_M h \,d\mu^{\delta} \bigg\}=0 \]
holds for all $h \in C^{\infty}(M)$.
Therefore we conclude
\begin{align*}
&\lim_{\delta \to 0} \frac{1}{\delta} \bigg\{
 \int_M \phi_{t+\delta} \,d\mu_{t+\delta} -\int_M \phi_t \,d\mu_t \bigg\} \\
&= \lim_{\delta \to 0} \frac{1}{\delta} \bigg\{
 \int_M (\phi_{t+\delta}-\phi_t) \,d\mu_{t+\delta} +\int_M \phi_t \,d\mu_{t+\delta}
 -\int_M \phi_t \,d\mu_t \bigg\} \\
&= \int_M \bigg\{ \frac{\del \phi_t}{\del t}
 +\frac{1}{m}\rho_t^{m-1}\Delta^{\omega}\phi_t
 +\frac{1}{m-1}\langle \nabla \phi_t,\nabla(\sigma^{m-1}) \rangle \bigg\} \,d\mu_t
\end{align*}
as desired.
$\qedd$
\end{proof}

\begin{remark}\label{rm:gfac}
In Theorem~\ref{th:gf}, assuming $\mu_t$ is absolutely continuous is redundant.
For $m>1$, $H_m(\mu_t|\nu)<\infty$ immediately implies $\mu_t \in \cP_{\ac}(M,\omega)$.
For $m<1$, if $\mu_t$ with $t>0$ has a nontrivial singular part $\mu^s$,
then the modification of $\mu_t$ as in the proof of Claim~\ref{cl:gf}(i) with
$\mu^{\delta}=\mu_t$ gives a measure
$\hat{\mu}_r \in \cP_{\ac}(M,\omega)$ for small $r>0$ such that
\[ W_2(\mu_t,\hat{\mu}_r)^2 \le \mu^s(M)r^2, \qquad
 H_m(\hat{\mu}_r|\nu) \le H_m(\mu_t|\nu) -C(\omega,m)\mu^s(M)r^{n(1-m)} \]
with $C>0$.
As $n(1-m)<1$, these yield $|\grad H_m(\cdot|\nu)|(\mu_t)=\infty$
as $r$ goes to zero, which contradicts the definition of gradient curves
(compare this with \cite[Theorem~10.4.8]{AGS}).
\end{remark}

Recall from Theorem~\ref{th:mCD} that the entropy $H_m(\cdot|\nu)$
is $K$-convex if (and only if) $\Ric_N \ge 0$ and $\Hess\Psi \ge K$.
Combining this with Theorems~\ref{th:cont}, \ref{th:gf},
we obtain the following.

\begin{corollary}\label{cr:gf}
Suppose that $(M,g)$ is compact and $\overline{M_0}$ is convex.
Then the weak solution $(\mu_t)_{t \in [0,\infty)} \subset \cP_{\ac}(\overline{M_0},\omega)$
to the porous medium $($or the fast diffusion$)$ equation \eqref{eq:pme}
constructed in Theorem~$\ref{th:gf}$ enjoys the $K$-contraction property $(\ref{eq:cont})$
under the assumptions $\Ric_N \ge 0$ and $\Hess\Psi \ge K$ on $\overline{M_0}$.
\end{corollary}

The argument in the proof of Theorem~\ref{th:gf} also shows that the absolute gradient
of $H_m(\cdot|\nu)$ at $\mu$ coincides with the square root of
the $m$-relative Fisher information introduced in $(\ref{eq:Im})$,
for general $m$.
Compare this with Theorem~\ref{th:LS}.

\begin{proposition}\label{pr:dHm}
Take $m \in [(n-1)/n,1) \cup (1,\infty)$ and
$\mu=\rho\omega \in \cP_{\ac}(M,\omega)$ such that $\rho$ is Lipschitz.
For any $(d^2/2)$-convex function $\varphi:M \lra \R$ and
the corresponding transport $\mu_t:=(\cT_t)_{\sharp}\mu$ with
$\cT_t(x):=\exp_x(t\nabla\varphi(x))$, $t \ge 0$, it holds that
\[ \lim_{t \to 0}\frac{H_m(\mu_t|\nu)-H_m(\mu|\nu)}{t} =\frac{1}{m-1}
 \int_M \langle \nabla(\rho^{m-1}-\sigma^{m-1}),\nabla\varphi \rangle
 \,d\mu. \]
In particular, we have $|\grad[H_m(\cdot|\nu)]|(\mu)=\sqrt{I_m(\mu|\nu)}$ and,
if $|\grad[H_m(\cdot|\nu)]|(\mu)<\infty$,
then the negative gradient vector $\grad[H_m(\cdot|\nu)](\mu)$ is achieved by
\[ \nabla\varphi =-\nabla \bigg( \frac{\rho^{m-1}-\sigma^{m-1}}{m-1} \bigg). \]
\end{proposition}

\begin{proof}
Recall that $\varphi$ is twice differentiable a.e., and that $\mu_t$ is
absolutely continuous for $t<1$ (\cite[Theorem~8.7]{Vi2}).
Using the calculation deriving $(\ref{eq:gf2})$, we obtain
\begin{align*}
&\lim_{t \to 0}\frac{H_m(\mu|\nu)-H_m(\mu_t|\nu)}{t} \\
&= \int_M \bigg\{ \frac{1}{m}\rho^{m-1}\Delta^{\omega}\varphi
 +\frac{1}{m-1}\langle \nabla\varphi,\nabla(\sigma^{m-1}) \rangle  \bigg\} \,d\mu \\
&= -\int_M \bigg\{ \frac{1}{m}\langle \nabla(\rho^m),\nabla\varphi \rangle
 -\frac{\rho}{m-1}\langle \nabla\varphi,\nabla(\sigma^{m-1}) \rangle  \bigg\}
 \,d\omega \\
&= -\frac{1}{m-1} \int_M \langle \nabla(\rho^{m-1}-\sigma^{m-1}),
 \nabla\varphi \rangle \,d\mu.
\end{align*}
As any geodesic with respect to $W_2$ is realized in this way (Theorem~\ref{th:Mc}),
we have $|\grad[H_m(\cdot|\nu)]|(\mu)=\sqrt{I_m(\mu|\nu)}$ and,
if $|\grad[H_m(\cdot|\nu)]|(\mu)<\infty$,
\[ \grad[H_m(\cdot|\nu)](\mu)
 =-\nabla \bigg( \frac{\rho^{m-1}-\sigma^{m-1}}{m-1} \bigg). \]
$\qedd$
\end{proof}

\begin{remark}\label{rm:pme}
The family of $m$-Gaussian measures (Example~\ref{ex:Nm})
is closely related to the Barenblatt solution to $(\ref{eq:pme})$ (without drift),
and again has a role to play here.
On the unweighted Euclidean space $(\R^n,dx)$, it is known by \cite[Proposition~5]{OW}
that a solution to \eqref{eq:pme} starting from an $m$-Gaussian measure
will keep being $m$-Gaussian.
An explicit expression of such solutions is given in \cite{Ta2}.
\end{remark}

\subsection{Proof of Claim~\ref{cl:gf}}

(i) The existence follows from, as usual, the compactness of $\cP(M)$
and the lower semi-continuity of $H_m(\cdot|\nu)$ (Lemma~\ref{lm:lsc}).
The absolute continuity is obvious for $m>1$.

For $m<1$, decompose $\mu^{\delta}$ into absolutely continuous and singular
parts $\mu^{\delta}=\rho\omega+\mu^s$ and suppose $\mu^s(M)>0$.
We modify $\mu^{\delta}$ into $\hat{\mu}_r \in \cP_{\ac}(M,\omega)$ as
\[ d\hat{\mu}_r(x)=\hat{\rho}_r(x) \,d\omega(x)
 :=\bigg\{ \rho(x) +\int_M \frac{\chi_{B(y,r)}(x)}{\omega(B(y,r))}
 \,d\mu^s(y) \bigg\} d\omega(x) \]
for small $r>0$.
Then we find
\begin{align*}
\int_M \sigma^{m-1} \,d\hat{\mu}_r
&\le \int_M \sigma^{m-1} \,d\mu^{\delta} +\int_M \bigg|
 \sigma(y)^{m-1}-\frac{1}{\omega(B(y,r))}\int_{B(y,r)} \sigma^{m-1} \,d\omega \bigg|
 d\mu^s(y) \\
&\le \int_M \sigma^{m-1} \,d\mu^{\delta}
 +\Big\{ \sup_M |\nabla(\sigma^{m-1})| \cdot r \Big\} \mu^s(M).
\end{align*}
Given an optimal coupling $\pi=\pi_1+\pi_2$ of $\mu^{\delta}$ and $\mu_t$
such that $(p_1)_{\sharp}\pi_1=\rho\omega$ and $(p_1)_{\sharp}\pi_2=\mu^s$,
\[ d\hat{\pi}_r(x,z):=d\pi_1(x,z)
 +\int_{y\in M} \frac{\chi_{B(y,r)}(x)}{\omega(B(y,r))} \,d\omega(x) d\pi_2(y,z) \]
is a coupling of $\hat{\mu}_r$ and $\mu_t$.
Hence we observe
\begin{align*}
W_2(\hat{\mu}_r,\mu_t)^2
&\le \int_{M \times M} d(x,z)^2 d\pi_1(x,z)
 +\int_{M \times M} \{ d(y,z)+r \}^2 \,d\pi_2(y,z) \\
&\le \int_{M \times M} d(x,z)^2 d\pi(x,z)
 +\{ 2\diam M +r \} r\pi_2(M \times M) \\
&\le W_2(\mu^{\delta},\mu_t)^2 +\{ 3\diam M \cdot r \} \mu^s(M).
\end{align*}
Finally, it follows from the H\"older inequality that
\begin{align*}
&\int_M \hat{\rho}_r^m \,d\omega
 = \int_M \bigg[ \int_M \bigg\{ \frac{\rho(x)}{\mu^s(M)}
 +\frac{\chi_{B(y,r)}(x)}{\omega(B(y,r))} \bigg\} \,d\mu^s(y) \bigg]^m d\omega(x) \\
&\ge \mu^s(M)^{m-1} \int_M \bigg[ \int_M \bigg\{ \frac{\rho(x)}{\mu^s(M)}
 +\frac{\chi_{B(y,r)}(x)}{\omega(B(y,r))} \bigg\}^m d\mu^s(y) \bigg] d\omega(x) \\
&\ge \mu^s(M)^{m-1} \int_M \bigg\{ \int_{M \setminus B(y,r)}
 \frac{\rho^m}{\mu^s(M)^m} \,d\omega +\int_{B(y,r)} \frac{1}{\omega(B(y,r))^m}
 \,d\omega \bigg\} d\mu^s(y) \\
&= \int_M \rho^m \,d\omega -\mu^s(M)^{-1} \int_M
 \bigg( \int_{B(y,r)} \rho^m \,d\omega \bigg) d\mu^s(y) \\
&\quad +\mu^s(M)^{m-1} \int_M \omega\big( B(y,r) \big)^{1-m} \,d\mu^s(y).
\end{align*}
As $M$ is compact, we find
\[ \mu^s(M)^{m-1} \int_M \omega\big( B(y,r) \big)^{1-m} \,d\mu^s(y)
 \ge \mu^s(M)^m C_1(\omega,m)r^{n(1-m)}, \]
and, for all $y \in \supp\mu^s$,
\begin{align*}
\int_{B(y,r)} \rho^m \,d\omega
&= \int_{B(y,r)} (\rho \sigma^{m-1})^m \sigma^{m(1-m)} \,d\omega \\
&\le \bigg( \int_{B(y,r)} \rho \sigma^{m-1} \,d\omega \bigg)^m
 \bigg( \int_{B(y,r)} \sigma^m \,d\omega \bigg)^{1-m} \\
&\le \bigg( \int_{B(y,r)} \rho \sigma^{m-1} \,d\omega \bigg)^m
 C_2(\omega,\sigma,m)r^{n(1-m)}.
\end{align*}
Since $\lim_{r \to 0} \sup_{y \in M} \int_{B(y,r)}\rho \sigma^{m-1} \,d\omega=0$,
these imply for small $r>0$
\[ \int_M \hat{\rho}_r^m \,d\omega
 \ge \int_M \rho^m \,d\omega +\frac{1}{2}C_1(\omega,m)\mu^s(M)^m r^{n(1-m)}. \]

Combining these, we conclude that
\begin{align*}
&H_m(\hat{\mu}_r|\nu) +\frac{W_2(\hat{\mu}_r,\mu_t)^2}{2\delta}
 -H_m(\mu^{\delta}|\nu) -\frac{W_2(\mu^{\delta},\mu_t)^2}{2\delta} \\
&\le -C_3(\omega,m) \mu^s(M)^m r^{n(1-m)}+C_4(M,\sigma,m,\delta)\mu^s(M)r,
\end{align*}
where $C_3,C_4>0$.
Then $n(1-m)<1$ and $\mu^s(M)>0$ yield that
\[ H_m(\hat{\mu}_r|\nu) +\frac{W_2(\hat{\mu}_r,\mu_t)^2}{2\delta}
 <H_m(\mu^{\delta}|\nu) +\frac{W_2(\mu^{\delta},\mu_t)^2}{2\delta} \]
holds for small $r>0$.
This contradicts the choice of $\mu^{\delta}$,
therefore we obtain $\mu^s(M)=0$.

(ii) By the choice of $\mu^{\delta}$, we have
\[ H_m(\mu^{\delta}|\nu) +\frac{W_2(\mu^{\delta},\mu_t)^2}{2\delta}
 \le H_m(\mu_t|\nu) \]
which immediately implies
$\lim_{\delta \to 0}W_2(\mu^{\delta},\mu_t)^2
 \le \lim_{\delta \to 0} 2\delta H_m(\mu_t|\nu)=0$.
Thus $\mu^{\delta}$ converges to $\mu_t$ weakly, and hence
\[ \limsup_{\delta \to 0} \frac{W_2(\mu^{\delta},\mu_t)^2}{2\delta}
\le H_m(\mu_t|\nu) -\liminf_{\delta \to 0} H_m(\mu^{\delta}|\nu) \le 0 \]
by the lower semi-continuity (Lemma~\ref{lm:lsc}).
This further yields
\[ H_m(\mu_t|\nu) \le \liminf_{\delta \to 0} H_m(\mu^{\delta}|\nu)
 \le \limsup_{\delta \to 0} H_m(\mu^{\delta}|\nu) \le H_m(\mu_t|\nu), \]
where the last inequality follows again from the choice of $\mu^{\delta}$.

(iii) This is a consequence of the following general lemma.
$\hfill \diamondsuit$

\begin{lemma}\label{lm:gf}
Assume $m \in [(n-1)/n,1) \cup (1,2]$ and that a sequence
$\{ \mu_i \}_{i \in \N} \subset \cP_{\ac}(M,\omega)$ converges to
$\mu \in \cP_{\ac}(M,\omega)$ weakly as well as
$\lim_{i \to \infty}H_m(\mu_i|\nu) =H_m(\mu|\nu) <\infty$.
Then, by setting $\mu_i=\rho_i \omega$ and $\mu=\rho \omega$,
$\rho_i^m$ converges to $\rho^m$ in $L^1(M,\omega)$.
\end{lemma}

\begin{proof}
Note that the convergence of $H_m(\mu_i|\nu)$ ensures
$\lim_{i \to \infty} \int_M \rho_i^m \,d\omega=\int_M \rho^m \,d\omega$.
We shall show the following:
\begin{itemize}
\item[$(*)$] For any constant $C>0$, it holds
$\lim_{i \to \infty}\|\min\{ \rho_i,C \} -\min\{ \rho,C \}\|_{L^2(M,\omega)}=0$.
\end{itemize}
Then we have, for $m<1$,
\[ \int_M |\rho_i^m -\rho^m| \,d\omega
 \le \int_M |\rho_i -\rho|^m  \,d\omega
 \le \omega(M)^{1-m} \bigg( \int_M |\rho_i -\rho| \,d\omega \bigg)^m, \]
and
\begin{align*}
&\int_M |\rho_i -\rho| \,d\omega \\
&\le \int_M \big[ |\min\{ \rho_i,C \} -\min\{ \rho,C \}|
 +\max\{ \rho_i-C,0 \} +\max\{ \rho-C,0 \} \big] \,d\omega \\
&\to 0
\end{align*}
as $i \to \infty$ and then $C \to \infty$.
Precisely,
\begin{align*}
\int_M \max\{ \rho_i -C,0 \} \,d\omega
&= \int_M (\rho_i -\min\{ \rho_i,C \}) \,d\omega
 \to 1-\int_M \min\{ \rho,C \} \,d\omega \quad (i \to \infty) \\
&\to 0 \quad (C \to \infty),
\end{align*}
where $(*)$ is used when taking the limit as $i \to \infty$.
For $m \in (1,2]$, we similarly find
\begin{align}
\int_M |\rho_i^m -\rho^m| \,d\omega
&\le m\int_M |\rho_i -\rho| \max\{\rho_i,\rho\}^{m-1} \,d\omega \nonumber\\
&\le m\bigg( \int_M |\rho_i -\rho|^m \,d\omega \bigg)^{1/m}
 \bigg(  \int_M (\rho_i +\rho)^m \,d\omega \bigg)^{(m-1)/m}, \label{eq:711}
\end{align}
and
\begin{align*}
&\int_M |\rho_i -\rho|^m \,d\omega \\
&\le 2^{m-1} \int_M \big[ |\min\{ \rho_i,C \} -\min\{ \rho,C \}|^m
 +\max\{ \rho_i-C,0 \}^m +\max\{ \rho-C,0 \}^m \big] \,d\omega \\
&\to 0
\end{align*}
as $i \to \infty$ and then $C \to \infty$.
Indeed,
\begin{align*}
\int_M \max\{ \rho_i -C,0 \}^m \,d\omega
&= \int_M (\rho_i -\min\{ \rho_i,C \})^m \,d\omega
 \le \int_M (\rho_i^m -\min\{ \rho_i,C \}^m) \,d\omega \\
&\to \int_M (\rho^m -\min\{ \rho,C \}^m) \,d\omega \quad (i \to \infty) \\
&\to 0 \quad (C \to \infty),
\end{align*}
where we used the calculation as in \eqref{eq:711} and $(*)$ to see
\[ \lim_{i \to \infty} \int_M |\min\{ \rho_i,C \}^m -\min\{ \rho,C \}^m| \,d\omega=0. \]

To show $(*)$, we suppose that it is false.
Then there are some constants $C,\ve>0$ and a sequence
$\{ l_j \}_{j \in \N} \subset \N$ going to infinity such that
\begin{equation}\label{eq:not*}
\| \min\{ \rho,C \} -\min\{ \rho_{l_j},C \} \|_{L^2(M,\omega)} \ge \ve
\end{equation}
for all $j \in \N$.
Note that, as $d^2[t^m/m(m-1)]/dt^2=t^{m-2}$,
\[ \frac{1}{m(m-1)} \bigg( \frac{\rho+\rho_{l_j}}{2} \bigg)^m \le
 \frac{\rho^m+\rho_{l_j}^m}{2m(m-1)}
 -\frac{\max\{ \rho,\rho_{l_j} \}^{m-2}}{8}|\rho-\rho_{l_j}|^2. \]
For the second term, we observe
\[ \max\{ \rho,\rho_{l_j} \}^{m-2} |\rho-\rho_{l_j}|^2
 \ge C^{m-2} |\min\{ \rho,C \} -\min\{ \rho_{l_j},C \}|^2. \]
This is clear if $\max\{ \rho,\rho_{l_j} \} \le C$ or
$\min\{ \rho,\rho_{l_j} \} \ge C$, and follows from
$\tau^{m-2}(\tau-\ve)^2 \ge (1-\ve)^2$ for $\tau \ge 1 \ge \ve$ otherwise.
Thus we obtain, by \eqref{eq:not*},
\begin{align*}
&\frac{1}{m(m-1)} \int_M
 \bigg( \frac{\rho+\rho_{l_j}}{2} \bigg)^m \,d\omega \\
&\le \int_M \frac{\rho^m+\rho_{l_j}^m}{2m(m-1)} \,d\omega
 -\frac{C^{m-2}}{8} \int_M |\min\{ \rho,C \} -\min\{ \rho_{l_j},C \}|^2
 \,d\omega \\
&\le \int_M \frac{\rho^m+\rho_{l_j}^m}{2m(m-1)} \,d\omega
 -\frac{C^{m-2}}{8} \ve^2.
\end{align*}
This means that $\bar{\mu}_j:=\{ (\rho+\rho_{l_j})/2 \} \omega$
satisfies
\[ \limsup_{j \to \infty} H_m(\bar{\mu}_j|\nu)
 \le \lim_{i \to \infty} H_m(\mu_i|\nu) -\frac{C^{m-2}}{8} \ve^2
 =H_m(\mu|\nu) -\frac{C^{m-2}}{8}\ve^2, \]
this contradicts the lower semi-continuity of $H_m(\cdot|\nu)$ (Lemma~\ref{lm:lsc}).
$\qedd$
\end{proof}

\section{Finsler case}\label{sc:Fins}

We finally stress that most results in this article are extended to Finsler manifolds,
according to the theory developed in \cite{Oint}, \cite{OS}
(see also a survey \cite{Osur}).
Briefly speaking, a Finsler manifold is a differentiable manifold equipped with
a (Minkowski) norm on each tangent space.
Restricting these norms to those coming from inner products,
we have the family of  Riemannian manifolds as a subclass.
We refer to \cite{BCS}, \cite{Sh} for the basics of Finsler geometry,
and to \cite{Oint}, \cite{OS}, \cite{Osur} for the details omitted
in the following discussion.

A {\it Finsler manifold} $(M,F)$ will be a pair of an $n$-dimensional
$C^{\infty}$-manifold $M$ and a $C^{\infty}$-Finsler structure $F:TM \lra [0,\infty)$
satisfying the following {\it regularity}, {\it positive homogeneity},
and {\it strong convexity} conditions:
\begin{enumerate}[(1)]
\item $F$ is $C^{\infty}$ on $TM \setminus 0$, where $0$ stands for
the zero section;
\item $F(\lambda v)=\lambda F(v)$ holds for all $v \in TM$ and $\lambda \ge 0$;
\item In any local coordinate system $(x^i)_{i=1}^n$ of an open set $U \subset M$ and
the corresponding coordinate $v=\sum_i v^i(\del/\del x^i)|_x$ of $T_xM$
with $x \in U$, the $n \times n$-matrix
\[ \bigg( \frac{\del^2(F^2)}{\del v^i \del v^j}(v) \bigg)_{i,j=1}^n \]
is positive-definite for all $v \in T_xM \setminus 0$ and $x \in U$.
\end{enumerate}
Then the distance $d$, geodesics and the exponential map are defined
in the same manner as Riemannian geometry, whereas $d$ is typically nonsymmetric
(and not a distance in the precise sense) since $F$ is merely positively homogeneous.
Nonetheless, $d$ satisfies the positivity and the triangle inequality.

On a Finsler manifold $(M,F)$, there is no constructive measure as good as
the Riemannian volume measure in the Riemannian case (cf.\ \cite{Odga}),
but we can consider an arbitrary positive $C^{\infty}$-measure $\omega$ on $M$
and associate it with the weighted Ricci curvature $\Ric_N$ (\cite{Oint}).
This curvature turns out  extremely useful, and the argument in \cite{Oint}
is applicable to generalizing the whole results in Sections~\ref{sc:dc}--\ref{sc:conc}
to the Finsler setting.
(We need a little trick only in Proposition~\ref{pr:Poin},
put $\mu=(1-\ve f)\sigma\omega$ when $m<1$ to have
$\nabla[((1-\ve f)^{m-1}-1)\sigma^{m-1}]=\nabla[(1-m)f\ve\sigma^{m-1}]
=(1-m)\ve \nabla(f\sigma^{m-1})$.)

\begin{theorem}\label{th:Fins}
Let $(M,F)$ be a forward complete, connected Finsler manifold and $\omega$ be
a positive $C^{\infty}$-measure on $M$.
Then the following results in this article hold true also for $(M,F,\omega)$
$($with appropriate interpretations for the nonsymmetric distance, cf.\
{\rm \cite{Oint}):}
\begin{itemize}
\item Equivalence between the convexity of $H_m(\cdot|\nu)$ and a curvature bound
 $($Theorem~$\ref{th:mCD});$
\item Functional inequalities
 $($Propositions~$\ref{pr:Ta}$, $\ref{pr:Poin}$, Theorem~$\ref{th:LS});$
\item Concentration of measures
 $($Theorem~$\ref{th:conc}$, Corollary~$\ref{cr:conc}$, Proposition~$\ref{pr:conc})$.
\end{itemize}
\end{theorem}

As for Section~\ref{sc:gf}, due to the lack of the analogue of
Theorem~\ref{th:angle}, we can not directly follow the Riemannian argument.
Nonetheless, we can apply the discussion in \cite{OS} using a (formal)
Finsler structure of the Wasserstein space, and obtain results corresponding
to Theorem~\ref{th:gf} and Proposition~\ref{pr:dHm}.
The point is the usage of the structure of the underlying space $M$,
while we did not explicitly use it in Subsections~\ref{ssc:gf1}, \ref{ssc:gf2}.
See \cite[Sections~6, 7]{OS} for further details.
We remark that, however, the $K$-contraction property \eqref{eq:cont}
essentially depends on the Riemannian structure and can not be expected
in the Finsler setting (cf.\ \cite{OS2}).

Let $(M,F)$ be compact from now on.
Due to Otto's idea \cite[Section~4]{Ot}, we introduce a Finsler structure
of $(\cP(M),W_2)$ as follows.
Given $\mu \in \cP(M)$, we define the {\it tangent space}
$(T_{\mu}\cP, F_W(\mu,\cdot))$ at $\mu$ by
\begin{align*}
F_W(\mu,\nabla\varphi) &:=\bigg( \int_M F(\nabla\varphi)^2 \,d\mu \bigg)^{1/2}
 \quad {\rm for}\ \varphi \in C^{\infty}(M), \\
T_{\mu}\cP &:=\Big( \overline{\{ \nabla\varphi \,|\, \varphi \in C^{\infty}(M) \}},
 F_W(\mu,\cdot) \Big),
\end{align*}
where the {\it gradient vector} $\nabla\varphi(x) \in T_xM$ is the Legendre transform
of the derivative $D\varphi(x) \in T^*_xM$,
and the closure was taken with respect to $F_W(\mu,\cdot)$.
We remark that the gradient $\nabla$ is a nonlinear operator
(i.e., $\nabla(\varphi_1+\varphi_2)(x) \neq \nabla\varphi_1(x)+\nabla\varphi_2(x)$
and $\nabla(-\varphi)(x) \neq -\nabla\varphi(x)$ in general),
since the Legendre transform is nonlinear unless $F|_{T_xM}$ is Riemannian.

Now, we take a locally Lipschitz curve $(\mu_t)_{t \in I} \subset (\cP(M),W_2)$
on an open interval $I \subset \R$.
We can associate it with the {\it tangent vector field}
$\dot{\mu}_t=\Phi(t,\cdot) \in T_{\mu_t}\cP$, that is,
$\Phi$ is a Borel vector field on $I \times M$ with $\Phi(t,x) \in T_xM$
and $F(\Phi) \in L^{\infty}_{\loc}(I \times M,d\mu_tdt)$ satisfying
the {\it continuity equation} $\del \mu_t/\del t +\div(\Phi_t \mu_t)=0$
in the weak sense that
\begin{equation}\label{eq:ceq}
\int_I \int_M \bigg\{ \frac{\del \phi_t}{\del t}+D\phi_t(\Phi_t) \bigg\} \,d\mu_t dt=0
\end{equation}
holds for all $\phi \in C^{\infty}_c(I \times M)$
(\cite[Theorem~8.3.1]{AGS}, \cite[Theorem~7.3]{OS}).
Using these `differentiable' structures, we can consider gradient curves
in a way different from the `metric' approach in Section~\ref{sc:gf}.

\begin{definition}\label{df:Fgf}
Given a function $f:\cP(M) \lra (-\infty,\infty]$ and $\mu \in \cP(M)$
with $f(\mu)<\infty$, we say that $f$ is {\it differentiable} at $\mu$
if there is $\Phi \in T_{\mu}\cP$ such that
\[ \lim_{t \downarrow 0}\frac{f((\cT_t)_{\sharp}\mu)-f(\mu)}{t}
 =\int_M \cL(\Phi)(\nabla\varphi) \,d\mu \]
holds for all $\varphi \in C^{\infty}(M)$, where $\cT_t(x):=\exp_x(t\nabla\varphi)$
and $\cL:T_xM \lra T_x^*M$ stands for the Legendre transform.
Then we write $\nabla_W f(\mu)=\Phi$.
\end{definition}

Then a gradient curve should be a solution to
$\dot{\mu}_t=\nabla_W[-H_m(\cdot|\nu)](\mu_t)$.
We first show that $\nabla_W[-H_m(\cdot|\nu)](\mu_t)$ is described by
the Fisher information like Proposition~\ref{pr:dHm}.

\begin{proposition}\label{pr:Fgf}
Take $\mu=\rho\omega \in \cP_{\ac}(M,\omega)$ with $\rho^m \in H^1(M,\omega)$.
If $\rho^{m-1}-\sigma^{m-1} \not\in H^1(M,\mu)$,
then $-H_m(\cdot|\nu)$ is not differentiable at $\mu$.
If $\rho^{m-1}-\sigma^{m-1} \in H^1(M,\mu)$,
then $-H_m(\cdot|\nu)$ is differentiable at $\mu$ and we have
\[ \nabla_W[-H_m(\cdot|\nu)](\mu)
 =\nabla\bigg( \frac{\rho^{m-1}-\sigma^{m-1}}{1-m} \bigg) \in T_{\mu}\cP. \]
\end{proposition}

\begin{proof}
Fix arbitrary $\varphi \in C^{\infty}(M)$ and put $\cT_t(x):=\exp_x(t\nabla\varphi(x))$,
$\mu_t=\rho_t\omega:=(\cT_t)_{\sharp}\mu$ for sufficiently small $t>0$.
Then the Jacobian equation $\rho=\rho_t(\cT_t)\bJ^{\omega}_t$ holds $\mu$-a.e.
(\cite[Theorem~5.2]{Oint}),
where $\bJ^{\omega}_t(x)$ stands for the Jacobian of
$D\cT_t(x):T_xM \lra T_{\cT_t(x)}M$ with respect to $\omega$.
Thus we obtain, as in the proof of Theorem~\ref{th:gf},
\[ H_m(\mu_t|\nu)=H_m(\mu|\nu) +\frac{1}{m(m-1)} \int_M
 \big[ \rho^{m-1} \{ (\bJ^{\omega}_t)^{1-m}-1 \}
 +m\{ \sigma^{m-1}-\sigma(\cT_t)^{m-1} \} \big] \,d\mu. \]
We observe, as $\rho^m \in H^1(M,\omega)$,
\begin{align*}
&\lim_{t \to 0} \int_M \frac{(\bJ^{\omega}_t)^{1-m}-1}{t} \rho^m \,d\omega
 = (1-m) \lim_{t \to 0} \int_M \frac{\bJ^{\omega}_t-1}{t} \rho^m \,d\omega \\
&=(1-m) \lim_{t \to 0} \int_M \frac{\rho^m-\rho(\cT_t)^m}{t} \bJ^{\omega}_t \,d\omega
 = (m-1) \int_M D(\rho^m)(\nabla\varphi) \,d\omega \\
&=m \int_M D(\rho^{m-1})(\nabla\varphi) \,d\mu,
\end{align*}
and hence
\[ \lim_{t \to 0}\frac{H_m(\mu|\nu)-H_m(\mu_t|\nu)}{t}
 =\int_M D\bigg( \frac{\rho^{m-1}-\sigma^{m-1}}{1-m} \bigg)
 (\nabla\varphi) \,d\mu. \]
This yields
\[ \nabla_W[-H_m(\cdot|\nu)](\mu)
 =\nabla\bigg( \frac{\rho^{m-1}-\sigma^{m-1}}{1-m} \bigg) \]
provided that $\rho^{m-1}-\sigma^{m-1} \in H^1(M,\mu)$.
Suppose $\rho^{m-1}-\sigma^{m-1} \not\in H^1(M,\mu)$.
Note that $\rho^{m-1}-\sigma^{m-1} \in L^2(M,\mu)$ since $\rho^m \in L^2(M,\omega)$
and $M$ is compact, thus we find $F(\nabla(\rho^{m-1}-\sigma^{m-1})) \not\in L^2(M,\mu)$.
Therefore we obtain
\[ \limsup_{\tilde{\mu} \to \mu}
 \frac{H_m(\mu|\nu)-H_m(\tilde{\mu}|\nu)}{W_2(\mu,\tilde{\mu})} =\infty \]
by approximating $\rho^{m-1}-\sigma^{m-1}$ with $\phi \in C^{\infty}(M)$ and
choosing $\varphi=\phi/(1-m)$.
Hence $H_m(\cdot|\nu)$ is not differentiable at $\mu$.
$\qedd$
\end{proof}

\begin{theorem}\label{th:Fgf}
Let $(\mu_t)_{t \in [0,\infty)} \subset \cP_{\ac}(M,\omega)$ be a continuous curve that is
locally Lipschitz on $(0,\infty)$, and assume that $\mu_t=\rho_t\omega$
satisfies $\rho_t^m \in H^1(M,\omega)$ as well as
$\rho_t^{m-1}-\sigma^{m-1} \in H^1(M,\mu_t)$ for a.e.\ $t \in (0,\infty)$.
Then we have
\[ \dot{\mu}_t=\nabla_W[-H_m(\cdot|\nu)](\mu_t) \]
at a.e.\ $t \in (0,\infty)$ if and only if $(\rho_t)_{t \in [0,\infty)}$ is
a weak solution to the reverse porous medium $($or fast diffusion$)$ equation of the form
\begin{equation}\label{eq:rpme}
\frac{\del \rho}{\del t} =-\div_{\omega}\bigg[ \rho\nabla
 \bigg( \frac{\rho^{m-1}-\sigma^{m-1}}{1-m} \bigg) \bigg].
\end{equation}
\end{theorem}

\begin{proof}
If $\dot{\mu}_t=\nabla_W[-H_m(\cdot|\nu)](\mu_t)$ holds for a.e.\ $t$, then
Proposition~\ref{pr:Fgf} yields
\[ \dot{\mu}_t=\nabla\bigg( \frac{\rho_t^{m-1}-\sigma^{m-1}}{1-m} \bigg)
 \qquad {\rm a.e.}\ t. \]
Thus it follows from the continuity equation $(\ref{eq:ceq})$ that
\[ \int_0^{\infty} \int_M \frac{\del \phi_t}{\del t} \,d\mu_t dt
 =-\int_0^{\infty} \int_M D\phi_t \bigg[ \nabla\bigg( 
 \frac{\rho_t^{m-1}-\sigma^{m-1}}{1-m} \bigg) \bigg] \,d\mu_t dt \]
for all $\phi \in C_c^{\infty}((0,\infty) \times M)$.
Therefore $\rho_t$ weakly solves $(\ref{eq:rpme})$.

Conversely, if $\rho_t$ is a weak solution to $(\ref{eq:rpme})$,
then the same calculation implies that
\[ \Phi_t=\nabla\bigg( \frac{\rho_t^{m-1}-\sigma^{m-1}}{1-m} \bigg) \]
satisfies the continuity equation $(\ref{eq:ceq})$.
Therefore Proposition~\ref{pr:Fgf} shows
$\dot{\mu}_t=\Phi_t=\nabla_W[-H_m(\cdot|\nu)](\mu_t)$.
$\qedd$
\end{proof}

We meant by the `reverse' porous medium (or fast diffusion) equation the equation
with respect to the {\it reverse} Finsler structure $\overleftarrow{F}(v):=F(-v)$.
As the gradient vector for $\overleftarrow{F}$ is written by
$\overleftarrow{\nabla}u=-\nabla(-u)$, $(\ref{eq:rpme})$ is indeed rewritten as
\[ \frac{\del \rho}{\del t} =\div_{\omega}\bigg[ \rho\overleftarrow{\nabla}
 \bigg( \frac{\rho^{m-1}-\sigma^{m-1}}{m-1} \bigg) \bigg]. \]

\end{document}